\documentclass[11pt,twoside]{article}

\usepackage{cite}
\usepackage{fancyhdr}
\usepackage[english]{babel} 
\usepackage{amsmath,amsthm,amssymb,amsfonts}
\usepackage{verbatim}
\usepackage{enumerate}
\usepackage{color}
\usepackage{hyperref}
\usepackage{graphicx}
\usepackage{multirow}
\usepackage[overload]{empheq}
\usepackage{epstopdf}
\usepackage{subfigure}
\usepackage{algorithm}
\usepackage[noend]{algpseudocode}
\usepackage{float}            
\usepackage{subfigure}       
\usepackage{float}
\floatstyle{boxed} 
\restylefloat{figure}
\usepackage{relsize}
\usepackage{eso-pic}  

\usepackage[T1]{fontenc}
\usepackage[ansinew]{inputenc}
\usepackage{lmodern}
\usepackage{hyperref}

\graphicspath{{./Figures/}}

\def \P{{\sf I\kern-1.5ptP}}

\newtheorem{proposition}{Proposition}
\newtheorem{theorem}[proposition]{Theorem}
\newtheorem{lemma}[proposition]{Lemma}
\newtheorem{corollary}[proposition]{Corollary}

\newtheorem{assumption}[proposition]{Assumption}
\newtheorem*{assumption*}{Assumption}

\newcommand{\vertiii}[1]{{\left\vert\kern-0.25ex\left\vert\kern-0.25ex\left\vert #1
		\right\vert\kern-0.25ex\right\vert\kern-0.25ex\right\vert}}

\newcommand{\R}{\mathbb{R}}

\newcommand{\etab}{{\boldsymbol{\eta}}}
\newcommand{\eb}{{\bf e}}
\newcommand{\gb}{{\bf g}}
\newcommand{\hb}{{\bf h}}
\newcommand{\fb}{{\bf f}}

\newcommand{\xb}{{\bf x}}

\newcommand{\zb}{{\bf z}}
\newcommand{\rb}{{\bf r}}

\newcommand{\vb}{{\bf v}}
\newcommand{\ub}{{\bf u}}

\newcommand{\X}{{\cal X}}
\newcommand{\Y}{{\cal Y}}

\newcommand{\BE}{\begin{equation}}
\newcommand{\EE}{\end{equation}}

\renewcommand{\S}{\mathbf{S}}

\date{}
\begin{document}
\title{Generalized Structure Preserving Preconditioners for Frame-Based Image Deblurring\footnote{This is a preprint.}}
\author{Davide Bianchi\footnote{Member of INdAM-GNCS Gruppo Nazionale per il Calcolo Scientifico.}\\
	Dipartimento di Scienza e Alta Tecnologia,
	\\		Universit\`a dell'Insubria, \\
	email: d.bianchi9@uninsubria.it
	\and
	Alessandro Buccini\footnote{Member of INdAM-GNCS Gruppo Nazionale per il Calcolo Scientifico. Partially founded by the Young Researcher Project ``Reconstruction of sparse data'' of the GNCS group of INdAM and by the AMIS (Algorithms e Models for Imaging Science) project of the Fondazione Sardegna.}\\
	Dipartimento di Matematica e Informatica,\\
	Universit\`a degli Studi di Cagliari, Italy.\\
	email: alessandro.buccini@unica.it
}

\maketitle
\begin{abstract}
	We are interested in fast and stable iterative regularization methods for image deblurring problems with space invariant blur. The associated coefficient matrix has a Block Toeplitz Toeplitz Blocks (BTTB) like structure plus a small rank correction depending on the boundary conditions imposed on the imaging model. In the literature, several strategies have been proposed in the attempt to define proper preconditioner for iterative regularization methods that involve such linear systems. Usually, the preconditioner is chosen to be a Block Circulant with Circulant Blocks (BCCB) matrix because it can be efficiently exploit Fast Fourier Transform (FFT) for any computation, including the (pseudo-)inversion.
	Nevertheless, for ill-conditioned problems, it is well known that BCCB preconditioners cannot provide a strong clustering of the eigenvalues.
	Moreover, in order to get an effective preconditioner, it is crucial to preserve the structure of the coefficient matrix.
	
	On the other hand, thresholding iterative methods have been recently successfully applied to image deblurring problems, exploiting the sparsity of the image in a proper wavelet domain. 
	Motivated by the results of recent papers, we combine a nonstationary preconditioned iteration with the modified linearized Bregman algorithm (MLBA) and proper regularization operators. 
	
	Several numerical experiments shows the performances of our methods in terms of quality of the restorations. 
\end{abstract}

\section*{Introduction}
In image deblurring we are concerned in reconstructing an approximation of an image from blurred
and noisy measurements. This process can be modeled by an integral equation of the form    

\begin{equation} \label{eq:model2}
g(x,y) = (\kappa*f)(x,y)= \int\limits_{-\infty}^{+\infty} \int\limits_{-\infty}^{+\infty} \kappa(x,y,x',y') f(x',y')~dx'~dy' + \eta(x,y),  \qquad(x,y) \in \Omega \subset\R^2,
\end{equation}
where $f : \R^2 \to \R$ is the original image and $g : \R^2 \to \R$ is the observed imaged which is obtained from a combination of a convolution operator, represented by the convolution kernel $\kappa: \R^4 \to \R$, and the add of some (unavoidable) noise $\eta : \R^2 \to \R$ coming from perturbations on the observed data, measurement errors and approximation errors, for example. By assuming the convolution kernel $\kappa$ to be compactly supported and considering the ideal case $\eta = 0$ equation \eqref{eq:model2} becomes
\begin{equation*}
g= K\cdot f,
\end{equation*}
where $K$ is a compact linear operator. In this contest, the convolution kernel $\kappa$ is generally called \textit{point spread function} (PSF) and if it is spatially invariant, as it is the case in many applications, then it assumes the following expression
$$
\kappa(x,y,x',y') = \kappa(x-x',y-y'), \qquad \mbox{with } \kappa: \R^2 \to \R.
$$ 

Considering an uniform grid, images are represented by their color intensities measured on the grid (pixels).  In this paper, for the sake of simplicity, we will deal only with square and gray-scale images, even if all the techniques presented here carry over to images of different sizes and colors as well. 

Collected images are available only in a finite region, the field of view (FOV),
and the measured intensities near the boundary are affected by
data which lie outside the FOV.
\begin{figure}
	\begin{center}
		\includegraphics[width=5cm]{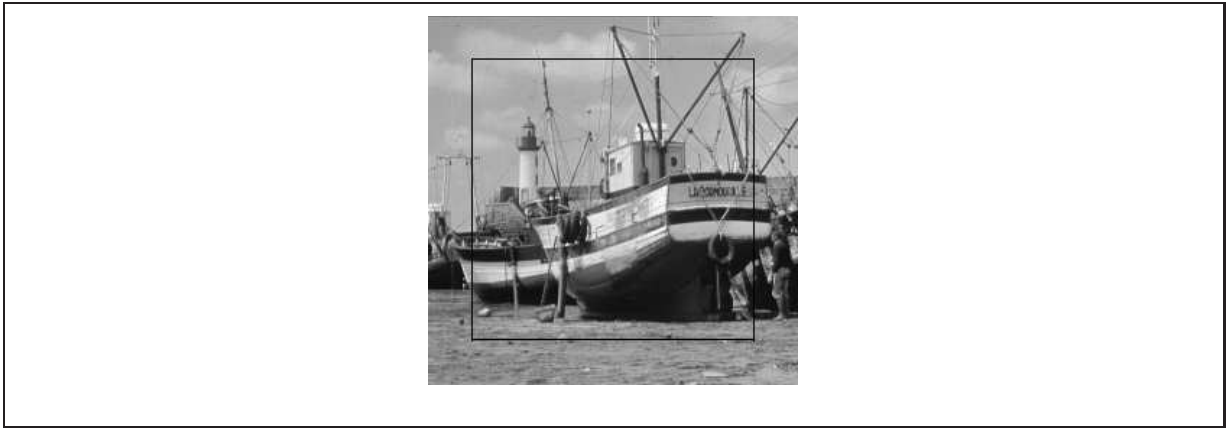}
	\end{center}
	\caption[FOV]{Field of view. We see what is inside the square box.}
\end{figure}
\\Denoting by $\gb$ and $\fb$ the stack ordered vectors corresponding to the observed image and the true image, respectively, the discretization of \eqref{eq:model2} by a rectangular quadrature rule with uniform grid (for example) leads to the under-determined linear system
\begin{equation} \label{eq:modeld}
\gb = K \fb + \etab,
\end{equation}
where the matrix $K$ is of size $m^2 \times k^2$. The matrix $K$ is often called the \textit{blurring matrix}. When imposing proper Boundary Conditions (BCs), the matrix $K$ becomes square $m^2 \times m^2$ and in some cases, depending on the BCs and the symmetry of the PSF, it can be diagonalized by discrete trigonometric transforms. Indeed, specific BCs induce specific matrix structures that can be exploited to lessen the computational costs using fast algorithms. Of course, since BCs are artificially introduced, their advantages could come with drawbacks in terms of reconstruction accuracy, depending on the type of problem.      
The BCs approach forces a functional dependency between the
elements of $\fb$ external to the FOV and those internal to this
area. 
If the BC model is not a good approximation of the real world outside the FOV,
the reconstructed image can be severely affected by some unwanted artifacts near the
boundary, called ringing effects; see, e.g., \cite{HNO05}.

\begin{figure}[!h]
	\begin{center}
		\begin{minipage}[c]{4.5cm}
			\centering
			\includegraphics[width=4cm]{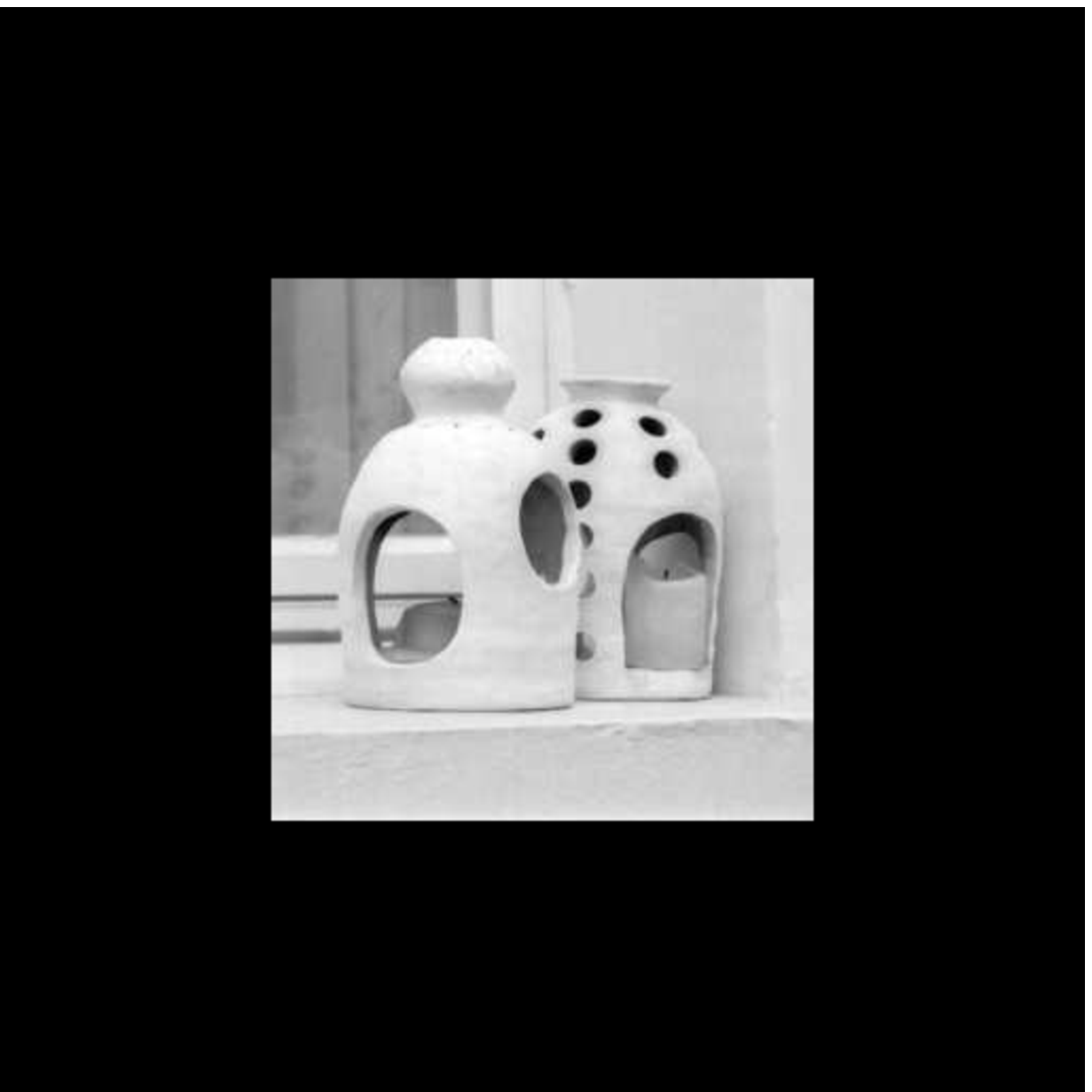}
			\small{zero Dirichlet}
		\end{minipage}\hspace{1cm}
		\begin{minipage}[c]{4.5cm}
			\centering
			\includegraphics[width=4cm]{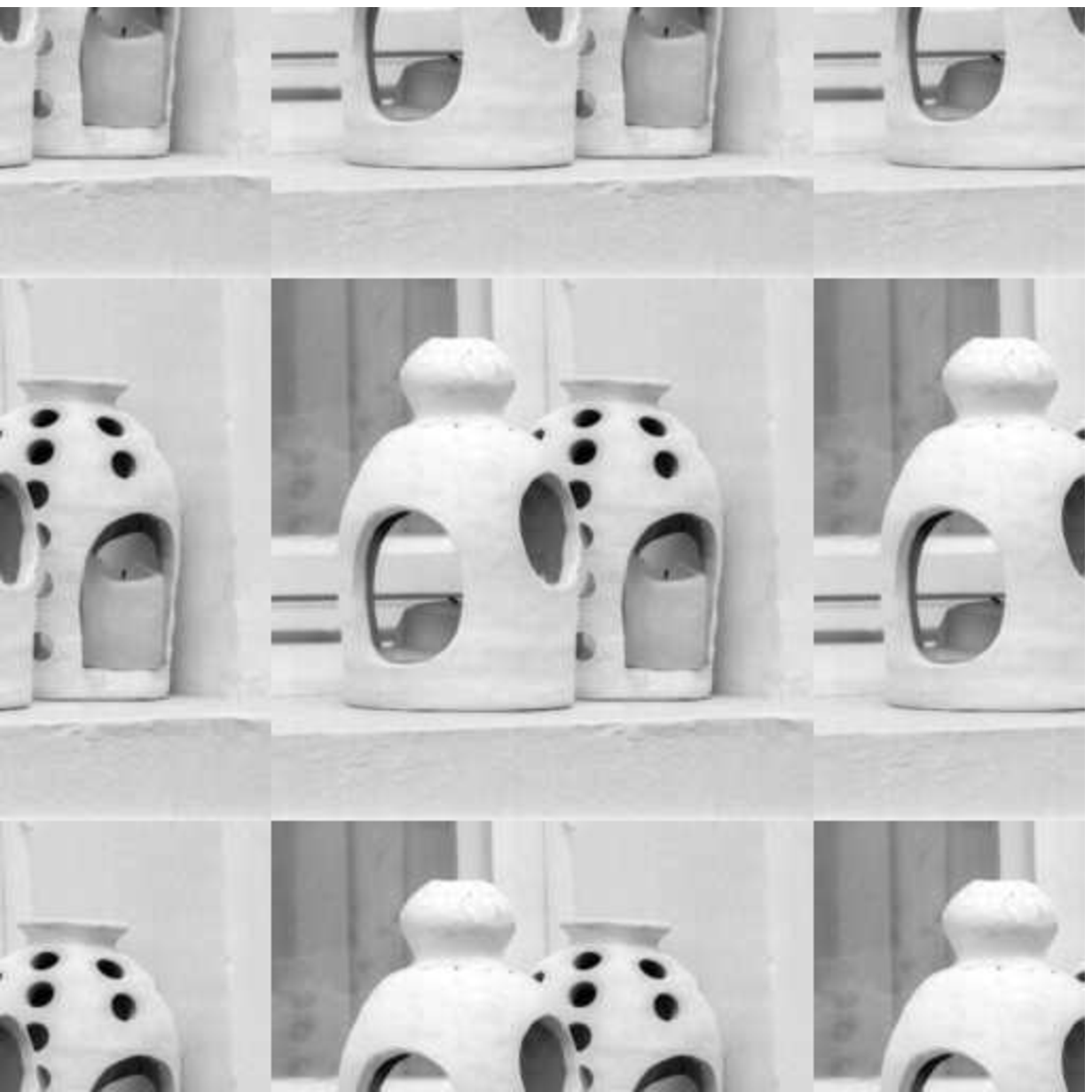}
			\small{Periodic}
		\end{minipage}
	\end{center}
	\vspace{0.1cm}
	\begin{center}
		\begin{minipage}[c]{4.5cm}
			\centering
			\includegraphics[width=4cm]{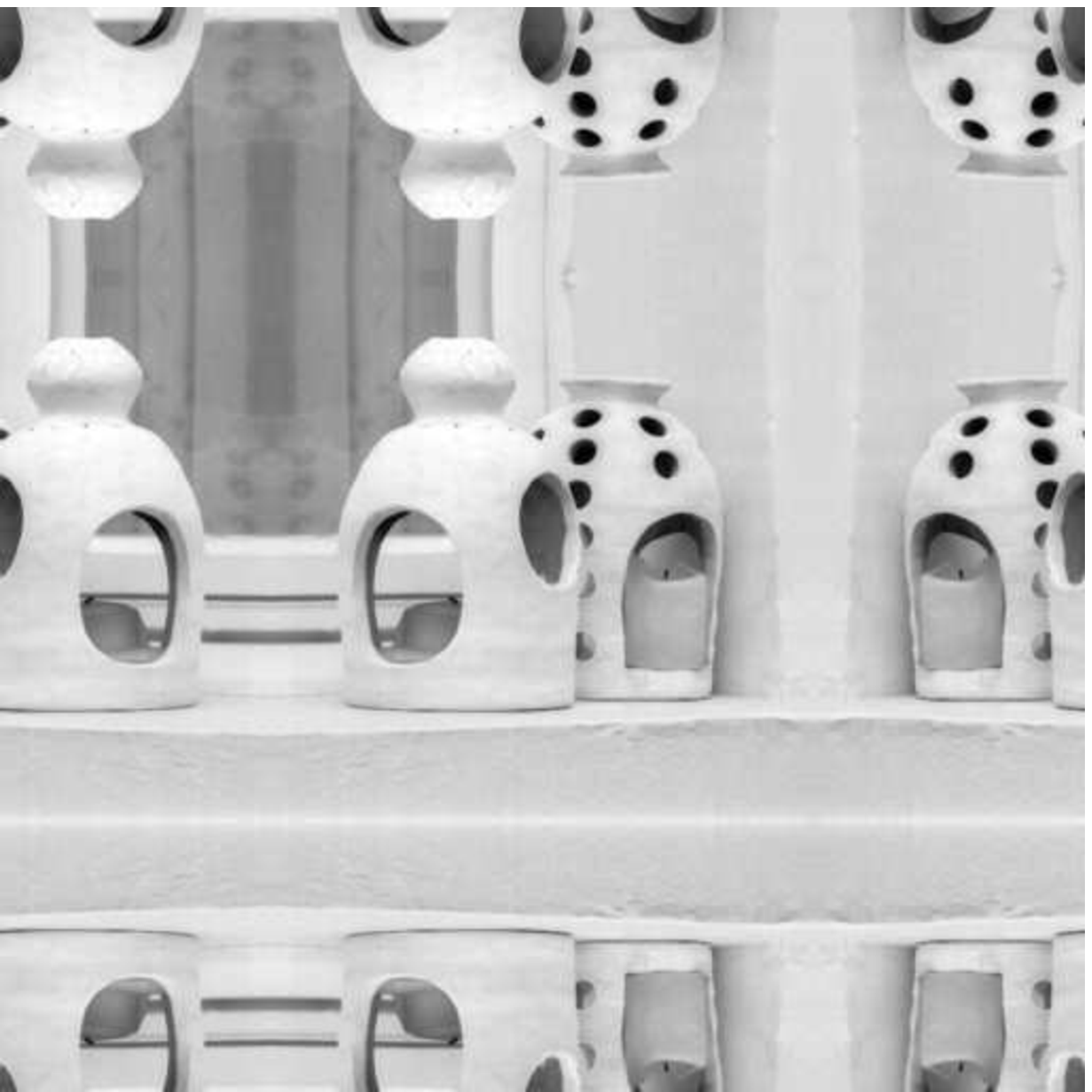}
			\small{Reflective}
		\end{minipage}\hspace{1cm}
		\begin{minipage}[c]{4.5cm}
			\centering
			\includegraphics[width=4cm]{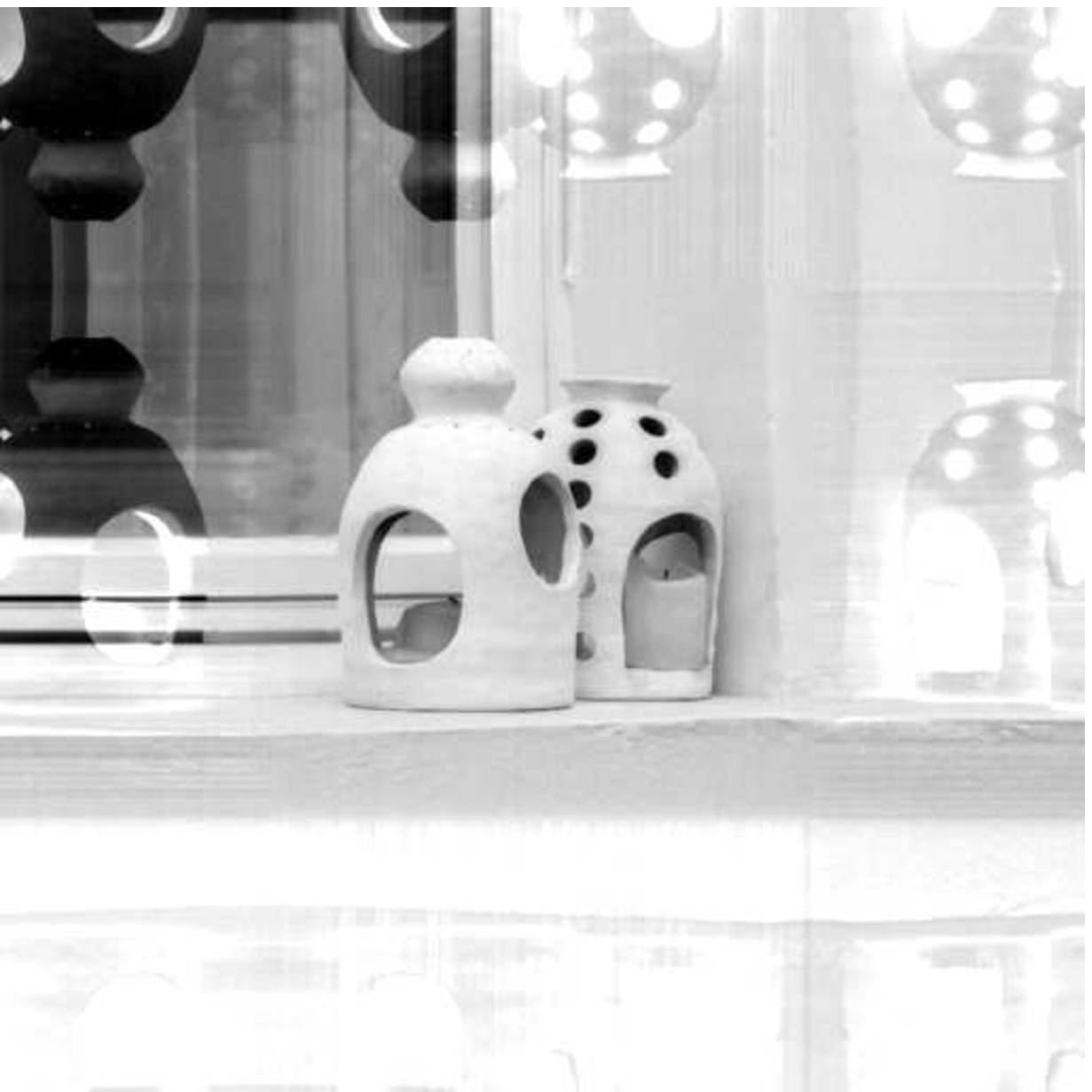}
			\small{Antireflective}
		\end{minipage}
		\caption[Boundary conditions]{Examples of boundary conditions.}\label{boundary_conditions}
	\end{center}
\end{figure}

The choice of the different BCs can be driven by some additional knowledge on the true image
and/or from the availability of fast transforms to diagonalize the matrix $K$ within $O(m^2\log(m))$ arithmetic operations.
Indeed, the matrix-vector product can be always computed by the 2D FFT, after a proper padding of the image to convolve, (see, e.g.,  \cite{RestoreTools}), while the availability of fast transforms to diagonalize the matrix $K$ depends on the BCs. Among the BCs present in the literature, we consider the following ones, but our approach can be extended to other BCs like, e.g., synthetic BCs \cite{AF13} or high order BCs \cite{D17,D}.
\begin{itemize}
	\item Zero (Dirichlet): the image outside the FOV is supposed to be null, i.e., zero pixel-valued. By using Zero BCs, the operator $K$ turns to be a Block Toeplitz with Toeplitz Blocks (BTTB) matrix. We will use the symbol $\mathcal{T}$
	to denote this class of matrices. The Zero BCs can be useful for some applications in astronomy, where an empty dark space surrounds a well located object. On the other hand, 
	they give rise to high ringing effects close to the boundary of the restored image in other classical imaging applications, where the background is not uniformly black.
	\item Periodic: the image inside the FOV is periodically repeated outside the FOV. By using Periodic BCs the operator $K$ turns to be Block Circulant with Circulant Blocks (BCCB) matrix. We will use the symbol $\mathcal{C}$
	to denote this class of matrices. Periodic BCs are computational favorable since the matrix $K$ can be easily diagonalized by Fast Fourier Transform (FFT),
	but in the restoration process they may also generate artifacts along the boundaries as well. 
	\item Reflective: in the Reflective BCs model the image inside the FOV is reflected outside the FOV, as there were a vertical mirror along each edge. That way, the pixel values across the boundary are extended so that the continuity of the	image is preserved at the boundary. The corresponding blurring matrix $K$ is then a Block Toeplitz with Toeplitz Block matrix plus a Block Hankel with Hankel Blocks matrix plus a Block Hankel with Toeplitz Blocks matrix plus  a Block Toeplitz with Hankel Blocks matrix, and can be diagonalized by Discrete Cosine Transform (DCT) when the PSF is symmetric; see \cite{NCT99}. We will use the symbol $\mathcal{R}$ to denote
	this class of matrices.
	\item Anti-Reflective: in the Anti-Reflective BCs model instead, the image inside the FOV is anti-reflected outside the FOV so that other than the continuity of the image at the boundary even the continuity of the normal derivatives are preserved at the boundary. The corresponding blurring matrix $K$ is a Block Toeplitz plus Hankel with
	Toeplitz plus Hankel blocks plus a low rank correction,
	which can be diagonalized by Anti-Reflective Transform (ART) when the PSF is symmetric; see \cite{Serra}.
	We will use the symbol $\mathcal{AR}$ to denote this class of matrices.
\end{itemize}
See Figure~\ref{boundary_conditions} for an illustration of the described BCs.

Both for Reflective and Anti-Reflective BCs a fast transform is available, but only if the
PSF is quadrantally symmetric, i.e. symmetric in both horizontal and vertical direction.
In all these cases, the matrix-vector product can be done in $O(m \log m)$ by FFT,
using a proper pad of the vector in agreement with the BCs imposed and then
performing a circulant convolution of double size.
In some cases, Reflective and Anti-Reflective BCs are even cheaper, since they require only real operations instead of complex ones without needing any padding; see \cite{ADS08}.

On the other hand, since equation \eqref{eq:modeld} is the product of the discretization of a compact operator, $K$ is severely ill-conditioned and may be singular. Such linear systems are commonly referred to as linear discrete ill-posed problems; see, e.g., \cite{HH} for a discussion.
Therefore a good approximation of $\fb$ cannot be obtained from the algebraic solution (e.g., the least-square solution) of
\eqref{eq:modeld}, but regularization methods are required. The basic idea of regularization is to replace the original ill-conditioned
problem with a nearby well-conditioned problem, whose solution approximates
the true solution. One of the popular regularization techniques is Tikhonov regularization and it amounts in solving
\begin{equation}\label{eq:tik}
\min_{\fb}\{\|K\fb - \gb\|_2^2 + \mu \|\fb\|_2^2\},
\end{equation}
where $\|\cdot \|_2$ denotes the vector $2$-norm and $\mu>0$ is a regularization parameter to be chosen. 
The first term in \eqref{eq:tik} is usually refereed to as fidelity term and the second as regularization term.
This approach is computationally attractive, since it leads to a linear problem and indeed several efficient methods have been developed for computing its solution and for estimating $\mu$; see \cite{HH}.
On the other hand, the edges of restored image are usually over-smoothed.
To overcome this unpleasant property, nonlinear strategies have been employed, like total variation (TV) \cite{ROF} and
thresholding iterative methods \cite{D3,FN03}. Anyway, several nonlinear regularization methods have an inner step that apply a least-square regularization and hence can benefit from strategies previously developed for such simpler model.

In the present paper, both the regularization strategies that we propose share two common ingredients: wavelet decomposition and $\ell_1$-norm minimization on the regularization term. This is motivated by the fact that most real images usually have sparse approximations under
some wavelet basis. In particular, in this paper we consider the tight frame systems previously used in \cite{CCSS03,COS,COS2}.
The redundancy of the tight frame system leads to robust signal representation in which partial loss of
the data can be tolerated without adverse effects. In order to obtain the sparse approximation, we minimize the weighted $\ell_1$-norm of the tight frame coefficients.
Let $W^*$ be a wavelet or tight-frame synthesis operator ($W^*W=I$), the
wavelets or tight-frame coefficients of the original image $\fb$ are $\xb$ such that
\begin{equation*}
\fb= W^*\xb, \qquad \mbox{and the blurring operator becomes } A=KW^*.
\end{equation*}

Within this frame set, the model equation \eqref{eq:modeld} translates into
\begin{equation}\label{eq:modeld2}
\gb = A\xb.
\end{equation} 
If we require to deal with positive (semi-)definite matrices, instead of the system \eqref{eq:modeld2}, we can solve the system of the normal equations
\begin{equation}\label{systemn}
A^{*}\gb= A^{*}A\xb,
\end{equation}
where $A^{*}$ is the conjugate transpose of $A$. This choice allows us to use many iterative methods, such as the iterated version of \eqref{eq:tik}, i.e., the iterated Tikhonov scheme, or the Conjugate Gradient (CG) and its generalizations.
Moreover, all iterative methods,
when are applied to the normal equations (\ref{systemn}),
become more stable, i.e. less sensitive with respect to data noise.
Unfortunately, in solving (\ref{systemn}) instead of (\ref{eq:modeld2}), the rate of convergence slows down. In this respect, the conventional technique to speed up the convergence is to consider the preconditioned system
\begin{equation*}
DA^{*}\gb=DA^{*}A\xb,
\end{equation*}
where $D$ is the so-called preconditioner, whose role is to suitably approximate the (generalized) inverse of the normal matrix $A^{*}A$ \cite{Piana}. In \cite{Acqua} it was proposed a new technique that uses a single preconditioning operator directly applied to the system \ref{eq:modeld2}. The new preconditioner, called as \textit{reblurring matrix} $P$, according to the terminology introduced in \cite{DS}, leads to the new preconditioned system
\begin{equation*}
P\gb=PA\xb .
\end{equation*}
As pointed out in \cite{Acqua}, the aim of the preconditioner $P$ is to allow iterative methods to become more stable (as well as usually obtained through the normal equations involving $A^{*}$) without slowing the convergence (so that no subsequent accelerating operator $D$ is needed), especially in the so-called signal space, i.e. the subspace less sensitive to the data noise. Combining this approach with a soft-thresholding technique such as the modified linearized Bregman splitting algorithm \cite{YOGD}, in order to mimic the $\ell_1$-norm minimization, leads to reformulate iterative methods as the Landweber method replacing the following preconditioned iterative scheme
\begin{equation*}
\xb_{n+1}=\xb_n + \tau DA^{*}(\gb-A\xb_n)
\end{equation*}
with
\begin{equation*}
\xb_{n+1}=\xb_n + \tau P(\gb-A\S_\mu(\xb_n)),
\end{equation*}
where $\tau$ is a positive relaxation parameter and $\S_\mu(\cdot)$ is the soft-thresholding function as defined in \ref{eq:soft}. In the following we fix $\tau=1$, by applying an implicit rescaling of the preconditioned system matrix $PA$.

The paper is organized as follows: in Section~\ref{sec:appTik} we propose a generalization of an approximated iterative Tikhonov scheme that was firstly introduced in \cite{DH13} and then developed and adapted into different settings in \cite{B,CDBH}. Here the preconditioner $P$ takes the form
$$
P= B^*\left(BB^* + \alpha_n \Lambda\Lambda^*\right)^{-1},
$$
where $B$ is an approximation of $A$, in the sense that $B=CW^*$ with $C$ the discretization of the same problem \eqref{eq:model2} as the original blurring matrix $K$ but imposing Periodic BCs. The operator $\Lambda\Lambda^*$ can be a function of $CC^*$ or the discretization of a differential operator. The method is nonstationary and the parameter $\alpha_n$ is computed by solving a nonlinear problem with a computational cost of $O(m^2)$. Related work on this kind of preconditioner can be found in \cite{BDR17,BPRXX,HRY16}. In Section~\ref{sec:struct} we define a class of preconditioners $P$ endowed with the same structure of the system matrix $A$, as initially proposed in \cite{DDEM} and then further developed in \cite{BBD}. It is called structure preserving reblurring preconditioning strategy and we combine it with the generalized regularization filtering approach of the preceding Section~\ref{sec:appTik}. The idea is to preserve both the informations carried over by the spectra of the operator $A$ and the structure itself of the operator induced by the best fitting BCs. 
Section~\ref{sec:exp} contains a selection of significant numerical examples which confirm the robustness and quality of the proposed regularization schemes. Section~\ref{sec:conclusions} provides a resume of the techniques presented in this work and draws some conclusions. Finally, in Appendix~\ref{appendix} are provided proofs of convergence and regularization properties of the proposed algorithms. 

\section{Preconditioned Iterated Soft-thresholding Tikhonov with general regularizing operator}\label{sec:appTik}
\subsection{Preliminary definitions}\label{sec:def}
Before proceeding further, let us introduce here some definitions and notations that will be used even in the forthcoming sections. We consider 
$$
K : \left(\R^{m^2}, \|\cdot \| \right) \to \left(\R^{m^2}, \|\cdot \| \right)
$$ 
to be the discretization of a compact linear operator
\begin{equation*}
\gb=K\fb,
\end{equation*}
where the Euclidean 2-norm $\|\cdot \|$ is induced by the standard Euclidean inner product  
$$
\langle \fb^{(1)}, \fb^{(2)} \rangle_{\R^{m^2}} =\sum_{j=1}^m f^{(1)}_jf^{(2)}_j.  
$$
Hereafter, we will specify the vector space where the inner product $\langle \cdot, \cdot \rangle$ acts only whenever it is necessary for disambiguation. The analysis that will follow in the next sections will be performed generally on a perturbed data $\gb^\delta$, namely
\begin{equation*}
\gb^\delta=K\fb ,
\end{equation*}
with $\gb^\delta = \gb + \etab$, and where $\etab$ is a noise vector such that $\|\etab\|=\delta$, $\delta$ is called the noise level.

Let 
$$
C : \left(\R^{m^2}, \|\cdot \| \right) \to \left(\R^{m^2}, \|\cdot \| \right)
$$  
be the discretization of a compact linear operator that approximates $A$, in a sense that will be specified later. Let 
$$
W : \left(\R^{m^2}, \|\cdot \| \right) \to \left(\R^{s}, \|\cdot \| \right)
$$
be such that 
$$
W^*W = I,
$$

where $W^* : \R^s \to \R^{m^2}$ indicates the adjoint operator of $W$, i.e.,  $\langle W\fb, \ub\rangle_{\R^{s}} = \langle \fb, W^*\ub\rangle_{\R^{m^2}}$ for each pair $\fb\in \R^{m^2}, \ub \in \R^s$. We define 
$$
\xb= W\fb, \qquad A= KW^*, \qquad B= CW^*.
$$
Let us introduce the following matrix norm. Given a generic linear operator
$$
L : \left(\R^s, \|\cdot\|_\infty\right) \to \left(\R^{m^2}, \|\cdot\|\right),
$$
where $\|\cdot \|_\infty$ is the sup norm, let us define the matrix norm $\vertiii{\cdot}$ as
\begin{equation}\label{def:norm}
\vertiii{L}:= \sup_{\|\xb \|_\infty \equiv 1} \|L\xb\|_2.
\end{equation}
Finally, let $\mu\geq 0$ and let $\S_\mu : \mathbb{R}^s \to \mathbb{R}^s$ be such that
\begin{equation}\label{eq:soft}
[\S_\mu(\mathbf{u})]_i=S_\mu(u_i),
\end{equation}
with $S_\mu$ the soft-thresholding function
\[
S_\mu(u_i)={\rm sgn}(u_i)\max\left\{|u_i|-\mu,\,0\right\}.
\]

\subsection{General regularization operator as $h(CC^*)$}
Let $h: [0, \|CC^*\|^2] \to \R$ be a continuous function such that 
$$
0<c_1\leq h(\sigma^2)\leq c_2.
$$
Define $c:= c_1/c_2$. We can now introduce the following algorithm. 
\begin{algorithm}
	\caption{$\mbox{PISTA}_h$}\label{P=f(CC^*) circulant}
	\begin{algorithmic}
		\State 	Fix $\zb^0\in \R^s$, $\delta >0$ and set $\xb^{0} = \S_\mu(\zb^0)$, $\rb^0 = \gb - A \xb^0$.
		\State  Set $\rho \in (0,c/2)$ and $q \in (2\rho,c)$.
		\State  Compute $\tau = \frac{1+ 2\rho}{c- 2\rho}$ and $\rb^n = \gb - A \xb^n$.
		\While{$\| \rb^n \| > \tau \delta$}
		\State Compute $\tau_n := \|\rb^n \|/ \delta$.
		\State Compute $q_n := \max \{ q, 2\rho + (1+\rho)/\tau_n \}$.
		\State Compute $\alpha_n$ such that
		\begin{eqnarray}\label{eq:Alg4_rev_7}
		\alpha_n \|(CC^* + \alpha_n h\left(CC^*\right))^{-1} \rb^n\| = \frac{q_n}{c_1}\|\rb^n\|.
		\end{eqnarray}
		\EndWhile 
		\State Compute 
		\begin{equation}\label{eq:hath}
		\hb^n = WC^* (CC^* + \alpha_n h\left(CC^*\right))^{-1} \rb^n.
		\end{equation}
		\State Compute 
		\begin{equation*}
		\left\{
		\begin{array}{l}
		\zb^{n+1}=\zb^{n}+ \hb^n, \\
		\xb^{n+1}=\S_\mu(\zb^{n+1}).
		\end{array}
		\right.
		\end{equation*}
	\end{algorithmic}
\end{algorithm}

A rigorous and full detailed analysis of the preceding algorithm will be performed in Appendix~\ref{appendix}. In order to prove all the desired properties we will need a couple of assumptions on the operators $K$ , $C$, and on the parameter $\mu$, that we present here below. 	
\begin{assumption}\label{hp:1}
	\begin{subequations}\label{eq:Alg4_rev_00}
		\begin{equation}\label{eq:Alg4_rev_0}
		\| \left( C - K \right) \fb \| \leq \rho \| K \fb \|, \qquad \forall \, \fb \in \R^{m^2},
		\end{equation}
		and
		\begin{equation}\label{eq:Alg4_rev_01}
		\mu \leq \frac{\rho\delta}{\vertiii{B}}, 
		\end{equation}
	\end{subequations}
	with a fixed $0 < \rho < c/2$, where $\delta=\|\etab\|$ is the noise level and where $\vertiii{\cdot}$ is the operator norm defined in \eqref{def:norm}.
\end{assumption}
Let us observe that Equation \eqref{eq:Alg4_rev_0} translates into
\begin{equation}\label{eq:Alg4_rev_1}
\| ( B - A ) \ub \| \leq \rho \| A \ub \|, \qquad \forall \, \ub \in \R^s.
\end{equation}
Let us spread some light on the preceding conditions. Assumption \eqref{eq:Alg4_rev_0}, or equivalently \eqref{eq:Alg4_rev_1}, is a strong assumption. It may be hard to satisfy it for every specific problem, as it implies
\begin{equation*}
\left(1-\rho\right)\|K\vb\|\leq \|C\vb\| \leq \left(1+\rho\right)\|K\vb\|  \qquad \mbox{for all } \vb \in \R^{m^2},
\end{equation*}
or equivalently 
\begin{equation}\label{eq:2c'}
\left(1-\rho\right)\|A\ub\|\leq \|B\ub\| \leq \left(1+\rho\right)\|A\ub\|  \qquad \mbox{for all } \ub \in \R^{s}, 
\end{equation}
that is, $K$ and $C$ are spectrally equivalent. Nevertheless, in image deblurring the boundary conditions have a very local effect, i.e., the approximation error $C-K$ can be decomposed as
$$
C-K = E+R,
$$  
where $E$ is a matrix of small norm (and the zero matrix if the PSF is compactly supported), and $R$ is a matrix of small rank, compared to the dimension of the problem. This suggests that Assumption \eqref{eq:Alg4_rev_0} needs to be satisfied only in a relatively small subspace, supposedly being a zero measure subspace. In particular only for every $\eb^{n}_\delta$, with $n\geq N$ and $N$ fixed, such that Proposition \ref{prop:Alg4_rev_1} could hold. All the numerical experiments are consistent with this observation but for a deeper understanding and a full treatment of this aspect we refer the reader to \cite[Section 4]{DH13}.  

On the other hand instead, Assumption \eqref{eq:Alg4_rev_01} is quite natural. It is indeed equivalent to require that
\begin{equation*}
\| B\left(\ub - \S_\mu (\ub)\right)\|\leq \rho\delta 
\end{equation*} 
that is, the soft-thresholding parameter $\mu = \mu(\delta)$ is continuously noise-dependent and it holds that $\mu(\delta) \to 0$ as $\delta \to 0$.

\subsection{General regularization operator as $\Lambda\Lambda^*$}
In image deblurring, in order to better preserve the edges of the reconstructed solution, it is usually introduced a differential operator $\Lambda\Lambda^*$, where $\Lambda: \X \to \Y$ is chosen as a first or second order differential operator which holds in its kernel all these functions which posses the key features of the true solution that we wish to preserve. In particular, since we are interested to recover the edges and curves of discontinuities of the true image, it is a common choice to rely on the Laplace operator with Neumann BCs, see \cite{EHN96}. In these recent papers \cite{BD,HS}, observing the spectral distribution of the Laplacian, it was proposed to substitute $\Lambda\Lambda^*$ with
$$
h(CC^*)= \left(I - \frac{CC^*}{\|CC^*\|}\right)^j,
$$
with $j \in \mathbb{N}$. 

Adding some new assumptions, we propose a modified version of the preceding Algorithm \ref{P=f(CC^*) circulant} that can take into account directly the operator $\Lambda$.

\begin{assumption}\label{hp:2}
	\begin{subequations}\label{eq:Alg4_rev_02}
		\begin{equation*}\label{eq:Alg4_rev_021}
		\textnormal{Ker}(K) \cap \textnormal{Ker}(\Lambda) = \{0\};
		\end{equation*}
		\begin{equation*}\label{eq:Alg4_rev_022}
		C_{|\textnormal{Ker}(\Lambda)} = K_{|\textnormal{Ker}(\Lambda)};
		\end{equation*}
		\begin{equation*}\label{eq:Alg4_rev_023}
		C \mbox{ and } \Lambda \mbox{ are diagonalized by the same unitary transform.}
		\end{equation*}
	\end{subequations}
\end{assumption}

\begin{algorithm}\label{alg:approx_lambda}
	\caption{$\mbox{PISTA}_\Lambda$}\label{P=LL* circulant}
	\begin{algorithmic}
		\State 	Fix $\zb^0\in \R^s$, $\delta >0$ and set $\xb^{0} = \S_\mu(\zb^0)$, $\rb^0 = \gb - A \xb^0$.
		\State  Set $\rho \in (0,1/2)$ and $q \in (2\rho,1)$.
		\State  Compute $\tau = \frac{1+ 2\rho}{1- 2\rho}$ and $\rb^n = \gb - A \xb^n$.
		\While{$\| \rb^n \| > \tau \delta$}
		\State Compute $\tau_n := \|\rb^n \|/ \delta$.
		\State Compute $q_n := \max \{ q, 2\rho + (1+\rho)/\tau_n \}$.
		\State Compute $\alpha_n$ such that
		\begin{eqnarray*}
		\alpha_n \|(CC^* + \alpha_n\Lambda\Lambda^*)^{-1} \rb^n\| = q_n\|\rb^n\|.
		\end{eqnarray*}
		\EndWhile 
		\State Compute 
		\begin{equation}\label{eq:hath_2}
		\hb^n = WC^* (CC^* + \alpha_n \Lambda\Lambda^*)^{-1} \rb^n.
		\end{equation}
		\State Compute 
		\begin{equation*}\label{eq:Alg4_revised_2}
		\left\{
		\begin{array}{l}
		\zb^{n+1}=\zb^{n}+ \alpha_n\hb^n, \\
		\xb^{n+1}=\S_\mu(\zb^{n+1}).
		\end{array}
		\right.
		\end{equation*}
	\end{algorithmic}
\end{algorithm}

We skip all the proofs of convergence since they can be recovered easily adapting the proofs in Section \ref{appendix} with \cite[Section 4]{B}.

\section{Structured PISTA with general regularizing operator}\label{sec:struct}
The structured case is a generalization of what developed in \cite{DDEM,BBD}, merging these ideas with the general approach described in Section \ref{sec:appTik}. We skip some details since they can be easily recovered from the aforementioned papers.

The creation of the blurring matrix $K$ is based on two ingredients: 	the PSF and the BCs enforced in the discretization. As already sketched in the Introduction, the latter choice gives rise to different types of structured matrices. For notational simplicity we consider a square PSF $H_\kappa \in \mathbb{R}^{k\times k}$ and we suppose that the position of the PSF center is known.

Given the pixels $\kappa_{i,j}$ of the PSF, it is possible to associate the so-called generating
	function $\kappa: \mathbb{R}^{2} \rightarrow \mathbb{C}$ as follows
	\begin{equation}\label{genfunct}
	\kappa(x_{1},x_{2}) 
	= \sum_{i,j=-m+1}^{m-1} \kappa_{i,j}\mathrm{e}^{\mathrm{\hat{\imath}}%
		(i x_{1}+j x_{2})}\text{ ,}
	\end{equation}
	where $\mathrm{\hat{\imath}}^{2}=-1$ and with the assumption that $\kappa_{i,j}=0$ if the element $(\kappa_{i,j})$ does not belong to $H_\kappa$ \cite{D}.
	Note that $\kappa_{j,j}$ are the Fourier coefficients of $\kappa\in {\rm span}\{\mathrm{e}^{\mathrm{\hat{\imath}}%
		(ix_{1}+jx_{2})}, i,j=-k,\dots,k\}$, so that the generating function $\kappa$
	contains the same information of $H$.

Summarizing the notation that we set in the Introduction about the BCs, we have
\begin{equation*}
	\begin{array}{ll}
	\text{Zero BCs:} & K=\mathcal{T}_{m}(\kappa), \\
	\text{Periodic BCs:} & K=\mathcal{C}_{m}(\kappa)=\mathcal{T}_{m}(\kappa)+\mathcal{B}_{m}^{\mathcal{C}}(\kappa), \\
	\text{Reflective BCs:} & K=\mathcal{R}_{m}(\kappa)=\mathcal{T}_{m}(\kappa)+\mathcal{B}_{m}^{\mathcal{R}}(\kappa), \\
	\text{Anti-Reflective BCs:} & K=\mathcal{AR}_{m}(\kappa)=\mathcal{T}_{m}(\kappa)+\mathcal{B}_{m}^{\mathcal{AR}}(\kappa).
	\end{array}
\end{equation*}
We notice that in all these four cases $K$ has a Toeplitz structure $\mathcal{T}_{m}(\kappa)$ which depends on $\kappa$ and given by the shift-invariant structure of the continuous operator, plus a correction term $\mathcal{B}_{m}^{\mathcal{X}}(\kappa),\, {\mathcal{X}} = {\mathcal{C}},{\mathcal{R}},{\mathcal{AR}}$ depending on the chosen BCs.
	
In conclusion, we employ the unified notation $K=\mathcal{M}_{m}(\kappa)$, where $\mathcal{M}(\cdot)$ can be any of the classes of matrices just introduced (i.e. $\mathcal{T}$, $\mathcal{C}$, $\mathcal{R}$, $\mathcal{AR}$). This notation highlights the two crucial ingredients that form $K$: the blurring phenomena associated with the PSF described by $\kappa$ and the involved BCs represented by $\mathcal{M}$.

Given the generating function $\kappa$ \eqref{genfunct} associated to the PSF $H_\kappa$, let us compute the eigenvalues $u_{i,j}$ of the corresponding BCCB matrix $\mathcal{C}_m(\kappa):=C$ by the means of a 2D-FFT, where $i,j=0,\cdots,m-1$. Fix a regularizing (differential) operator $\Lambda\Lambda^*$ as in Section~\ref{sec:appTik}, and suppose that the Assumptions~\ref{hp:1}~and~\ref{hp:2} holds. The differential operator can be of the form $\Lambda\Lambda^*=h(CC^*)$, as in Algorithm~\ref{P=f(CC^*) circulant} as well. Let now 
\begin{equation*}
v_{i,j} = \frac{\overline{u}_{i,j}}{|u_{i,j}|^2 + \alpha_n|\sigma_{i,j}|^2},
\end{equation*}
be the new eigenvalues after the application of the Tikhonov filter to $u_{i,j}$, where $\sigma_{i,j}$ are the eigenvalues (singular values) of $\Lambda$ and $\alpha_n$ is computed as in Algorithm \ref{P=f(CC^*) circulant}-\ref{P=LL* circulant}. Let us compute now the coefficients $\hat{\kappa}_{i, j}$ of
\begin{equation}\label{genfunct2}
\hat{\kappa}(x_1,x_2)=\sum_{i,j=-m+1}^{m-1}\hat{\kappa}_{i,j}\rm{e}^{\mathrm{\hat{\imath}}(ix_1+jx_2)}
\end{equation} 
by the means of a 2D-iFFT and, finally, let us define 
$$
P = \mathcal{M}_m(\hat{\kappa}),
$$
where $\mathcal{M}(\cdot)$ corresponds to the most fitting BCs for the model problem \eqref{eq:model2}.

We are ready now to formulate the last algorithm.

\begin{algorithm}
	\caption{$\mbox{Struct-PISTA}_\Lambda$}\label{struc1}
	\begin{algorithmic}
		\State  Fix $H_\kappa$, BCs, $\Lambda$.
		\State Set $C=\mathcal{C}_m(\kappa)$.
		\State  Get $\{u_{i,j}\}_{i,j=0}^{n-1}$ by computing an FFT of $H_\kappa$.
		\State 	Fix $\zb^0\in \R^s$, $\delta >0$ and set $\xb^{0} = \S_\mu(\zb^0)$, $\rb^0 = \gb - K \xb^0$.
		\State  Set $\rho \in (0,1/2)$ and $q \in (2\rho,1)$.
		\State  Compute $\tau = \frac{1+ 2\rho}{1- 2\rho}$ and $\rb^n = \gb - K \xb^n$.
		\While{$\| \rb^n \| > \tau \delta$}
		\State Compute $\tau_n := \|\rb^n \|/ \delta$.
		\State Compute $q_n := \max \{ q, 2\rho + (1+\rho)/\tau_n \}$.
		\State Compute $\alpha_{n}$ such that
		\begin{eqnarray*}
		\alpha_n \|(CC^* + \alpha_n \Lambda\Lambda^*)^{-1} \rb^n\| = q_n\|\rb^n\|.
		\end{eqnarray*}
		\EndWhile 
		\State Compute $v_{i,j}=\frac{\overline u_{i,j}}{\left\vert u _{i,j}\right\vert^2+\alpha_{n}|\sigma_{i,j}|^2}$.
		\State Get the mask $\widetilde{H}$ of the coefficients $\hat{\kappa}_{i,j}$ of $\hat{\kappa}$ of \eqref{genfunct2} by computing an IFFT of $\{v_{i,j}\}_{i,j=0}^{m-1}$.
		\State Generate the matrix $P:=\mathcal{M}_m(\hat{\kappa})$ from the coefficient mask $\widetilde{H}$ and BCs.
		\State Compute 
		\begin{equation}\label{h_struc}
		\hb^n = P\rb^n.
		\end{equation}
		\State Compute 
		\begin{equation*}
		\left\{
		\begin{array}{l}
		\zb^{n+1}=\zb^{n}+ \hb^n, \\
		\xb^{n+1}=\S_\mu(\zb^{n+1}).
		\end{array}
		\right.
		\end{equation*}
	\end{algorithmic}
\end{algorithm}
In the case that $\Lambda\Lambda^*=h(CC^*)$, then the algorithm is modified in the following way:
$$
\rho \in (0,c/2), \quad q \in (2\rho,c), \quad \alpha_n \|(CC^* + \alpha_n \Lambda\Lambda^*)^{-1} \rb^n\| = \frac{q_n}{c_1}\|\rb^n\|
$$
where $0<c_1\leq h(\sigma^2)\leq c_2$, $c:=c_1/c_2$. We will denote this version by $\mbox{Struct-PISTA}_h$. We will not provide a direct proof of convergence for this last algorithm. Let us just observe that the difference between \eqref{h_struc} and \eqref{eq:hath_2}-\eqref{eq:hath} is just a correction of small rank and small norm.

\section{Numerical experiments}\label{sec:exp}
We now compare the proposed algorithms with some methods from the literature. In particular, we consider the AIT-GP algorithm described in \cite{B} and the ISTA algorithm described in \cite{D3}. The AIT-GP method can be seen as Algorithm \ref{P=LL* circulant} with $\mu=0$, while the ISTA algorithm is equivalent to iterations of Algorithm~2 without the preconditioner. These comparisons allow us to show how the quality of the reconstructed solution is improved by the presence of both the soft-thresholding and the preconditioner.

The ISTA method and our proposals require the selection of a regularization parameter. For all these methods we select the parameter that minimizes the relative restoration error defined by
$$
{\rm RRE}(\mathbf{f})=\frac{\|\mathbf{f}-\mathbf{f}_{\rm true}\|}{\|\mathbf{f}_{\rm true}\|}.
$$ 
For the comparison of the algorithms we consider the Peak Signal to Noise Ratio (PSNR) defined by
$$
{\rm PSNR}(\mathbf{f})=20\log_{10}\left(\frac{mM}{\|\mathbf{f}-\mathbf{f}_{\rm true\|}}\right),
$$
where $m^2$ is the the number of elements of $\mathbf{f}$ and $M$ denotes the maximum value of $\mathbf{f}_{\rm true}$. Moreover, we consider the Structure SIMilarity index (SSIM), the definition of the SSIM is involved, here we recall that this index measures how accurately the computed approximation is able to reconstruct the overall structure of the image. The higher the value of the SSIM the better the reconstruction is, and the maximum value achievable is $1$; see \cite{SSIM} for a precise definition of the SSIM.

We now describe how we construct the operator $W$. We use the tight frames determined by linear B-splines; see, e.g., \cite{BR19}. For one-dimensional problems they are composed by a low-pass filter $W_0\in\R^{m\times m}$ and two high-pass filters $W_1\in\R^{m\times m}$ and $W_2\in\R^{m\times m}$. These filters are determined by the masks are given by
\begin{equation*}
u^{(0)}=\frac{1}{4}[1,2,1],\quad u^{(1)}=\frac{\sqrt{2}}{4}[1,0,-1],\quad 
u^{(2)}=\frac{1}{4}[-1,2,-1].
\end{equation*}
Imposing reflexive boundary conditions we determine the analysis operator $W$ so that $W^*W=I$. Define the matrices
\begin{equation*} W_{0}=\frac{1}{4}\left(\begin{array}{ccccc}
3 & 1 & 0& \dots &0\\
1 & 2 & 1  \\
&  \ddots& \ddots & \ddots \\
& &  1 & 2 & 1 \\
0& \dots & 0 & 1 &3
\end{array}\right)
,\quad
W_{1}=\frac{\sqrt2}{4}\left(\begin{array}{ccccc}
-1 & 1 & 0& \dots &0\\
-1 & 0 & 1  \\
&  \ddots& \ddots & \ddots \\
& &  -1 & 0 & 1 \\
0& \dots & 0 & -1 &1
\end{array}\right),
\end{equation*}
and 
\begin{equation*}
W_{2}=\frac{1}{4}\left(\begin{array}{ccccc}
1 & -1 & 0& \dots &0\\
-1 & 2 & -1  \\
&  \ddots& \ddots & \ddots \\
& &  -1 & 2 & -1 \\
0& \dots & 0 & -1 &1
\end{array}\right).
\end{equation*}
Then the operator $W$ is defined by
\begin{equation*}
W=\left(\begin{array}{c}
W_0\\W_1\\W_2
\end{array}\right).
\end{equation*}

To construct the two-dimensional framelet analysis operator we use the tensor products 
\begin{equation*}
W_{i,j}=W_{i}\otimes W_{j}, \quad i,j=0,1,2.
\end{equation*}
The matrix $W_{00}$ is a low-pass filter; all the other matrices $W_{ij}$ contain at least one high-pass filter. The analysis operator is given by
\begin{equation*}
W=\left[\begin{array}{c}
W_{00}\\
W_{01}  \\
\vdots\\
W_{22} \\
\end{array}
\right].
\end{equation*}

In $\mbox{PISTA}_h$, following \cite{HS}, we set 
$$
h(x)=\left(1-\frac{x}{\|A\|^2}\right)^4+10^{-15}.
$$

All the computations are performed on MATLAB R2018b running on a laptop with an Intel i7-8750H @2.20 GHz CPU and 16GB of RAM.

\paragraph{Cameraman}
We first consider the cameraman image in Figure~\ref{fig:cameraman}(a) and we blur it with the non-symmetric PSF in Figure~\ref{fig:cameraman}(b). We then add $2\%$ white Gaussian noise obtaining the blurred and noisy image in Figure~\ref{fig:cameraman}(c). Note that we crop the boundaries of the image to simulate real data; see \cite{HNO05} for more details. Since the image is generic we impose reflexive BCs.

In Table~\ref{tbl:compare} we report the results obtained with the different methods. We can observe that $\mbox{Struct-PISTA}_h$ provides the best reconstruction of all considered algorithms. Moreover, we can observe that, in general, the introduction of the structured preconditioner improves the quality of the reconstructed solutions, especially in term of SSIM. From the visual inspection of the reconstructions in Figure~\ref{fig:cameraman_rec} we can observe that the introduction of the structured preconditioner allows us to evidently reduce the boundary artifacts as well as avoid the amplification of the noise.

\begin{figure}
	\centering
	\begin{minipage}[c]{0.3\textwidth}
		\centering
		\includegraphics[width=\textwidth]{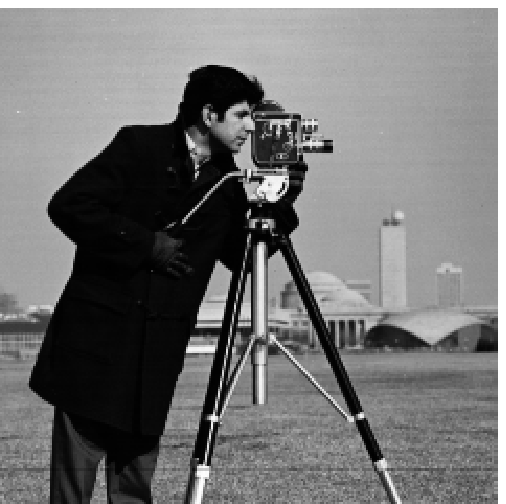}\\(a)
	\end{minipage}
	\begin{minipage}[c]{0.3\textwidth}
		\centering
		\includegraphics[width=\textwidth]{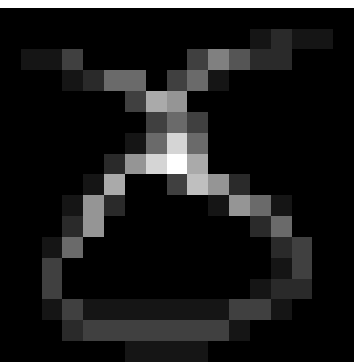}\\(b)
	\end{minipage}
	\begin{minipage}[c]{0.3\textwidth}
		\centering
		\includegraphics[width=\textwidth]{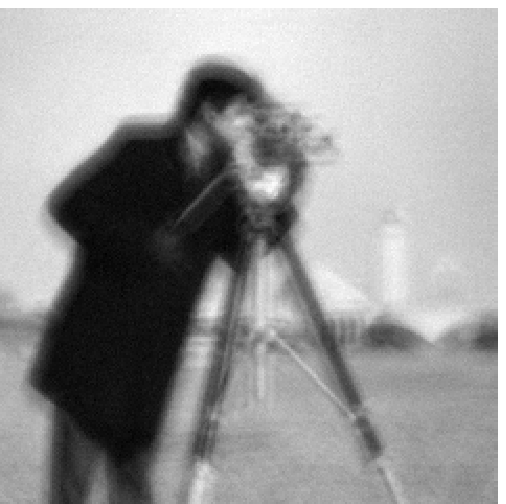}\\(c)
	\end{minipage}
	\caption{Cameraman test problem: (a) True image ($238\times 238$ pixels), (b) PSF ($17\times 17$ pixels), (c) Blurred and noisy image with $2\%$ of white Gaussian noise ($238\times 238$ pixels).}
	\label{fig:cameraman}
\end{figure}

\begin{figure}
	\centering
	\begin{minipage}[c]{0.3\textwidth}
		\centering
		\includegraphics[width=\textwidth]{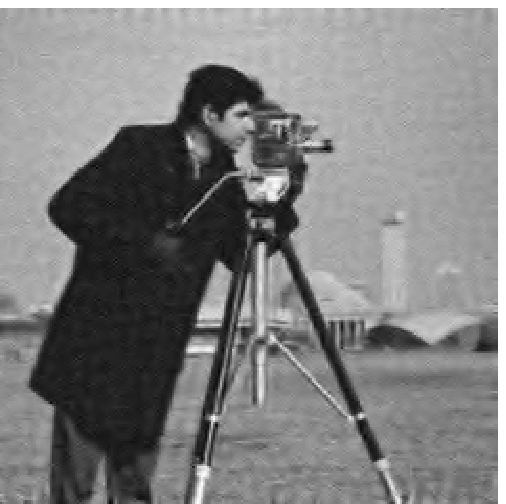}\\(a)
	\end{minipage}
	\begin{minipage}[c]{0.3\textwidth}
		\centering
		\includegraphics[width=\textwidth]{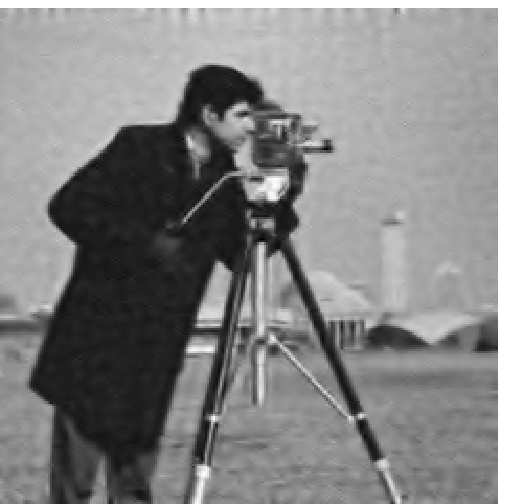}\\(b)
	\end{minipage}
	\begin{minipage}[c]{0.3\textwidth}
		\centering
		\includegraphics[width=\textwidth]{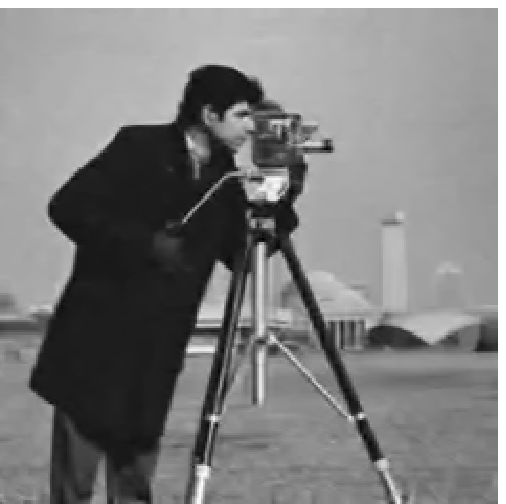}\\(c)
	\end{minipage}
	\caption{Cameraman test problem reconstructions: (a) ISTA, (b) $\mbox{PISTA}_h$, (c)  $\mbox{Struct-PISTA}_h$.}
	\label{fig:cameraman_rec}
\end{figure}

\paragraph{Grain}
We now consider the grain image in Figure~\ref{fig:grain}(a) and blur it with the PSF, obtained by the superposition of two motions PSF, in Figure~\ref{fig:grain}(b). After adding $3\%$ of white Gaussian noise and cropping the boundaries we obtain the blurred and noisy image in Figure~\ref{fig:grain}(c). According to the nature of the image we use reflexive bc's. 

Again in Table~\ref{tbl:compare} we report all the results obtained with the considered methods. In this case ISTA provides the best reconstruction in terms of RRE and PSNR. However, $\mbox{Struct-PISTA}_h$ provides  the best reconstruction terms of SSIM and very similar results in term of PSNR and RRE. In Figure~\ref{fig:grain_rec} we report some of the reconstructed solution. From the visual inspection of these reconstruction we can see that the introduction of the structured preconditioner reduces the ringing and boundary effects in the computed solutions. 

\begin{figure}
	\centering
	\begin{minipage}[c]{0.3\textwidth}
		\centering
		\includegraphics[width=\textwidth]{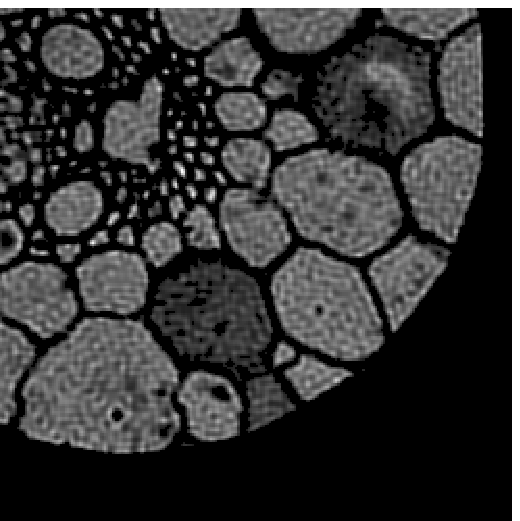}\\(a)
	\end{minipage}
	\begin{minipage}[c]{0.3\textwidth}
		\centering
		\includegraphics[width=\textwidth]{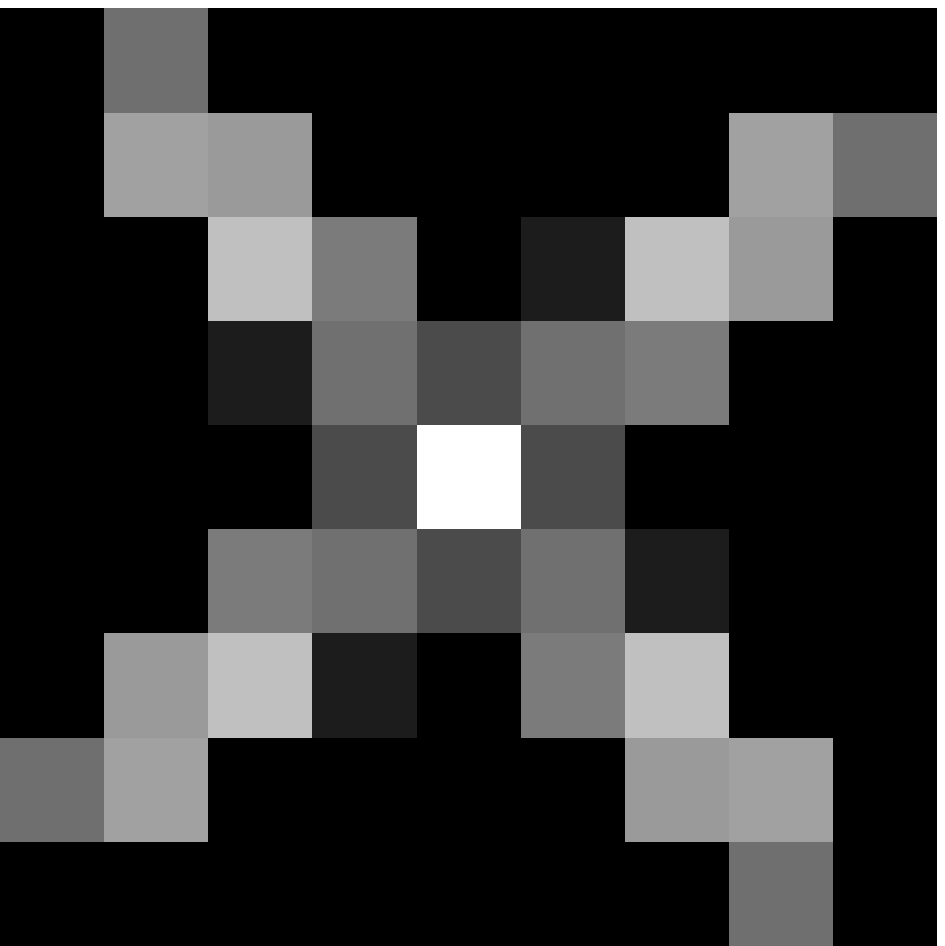}\\(b)
	\end{minipage}
	\begin{minipage}[c]{0.3\textwidth}
		\centering
		\includegraphics[width=\textwidth]{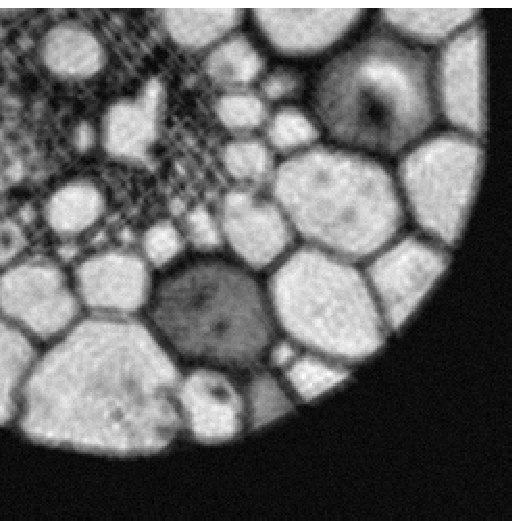}\\(c)
	\end{minipage}
	\caption{Grain test problem: (a) True image ($246\times 246$ pixels), (b) PSF ($9\times 9$ pixels), (c) Blurred and noisy image with $3\%$ of white Gaussian noise ($246\times 246$ pixels).}
	\label{fig:grain}
\end{figure}

\begin{figure}
	\centering
	\begin{minipage}[c]{0.3\textwidth}
		\centering
		\includegraphics[width=\textwidth]{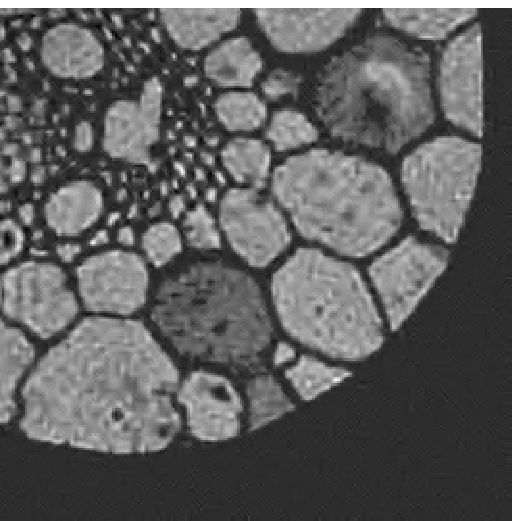}\\(a)
	\end{minipage}
	\begin{minipage}[c]{0.3\textwidth}
		\centering
		\includegraphics[width=\textwidth]{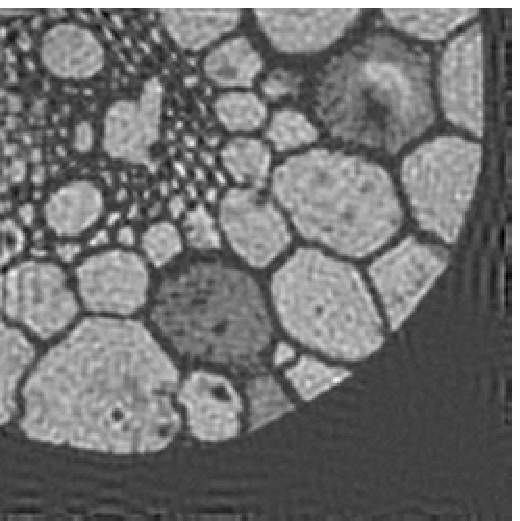}\\(b)
	\end{minipage}
	\begin{minipage}[c]{0.3\textwidth}
		\centering
		\includegraphics[width=\textwidth]{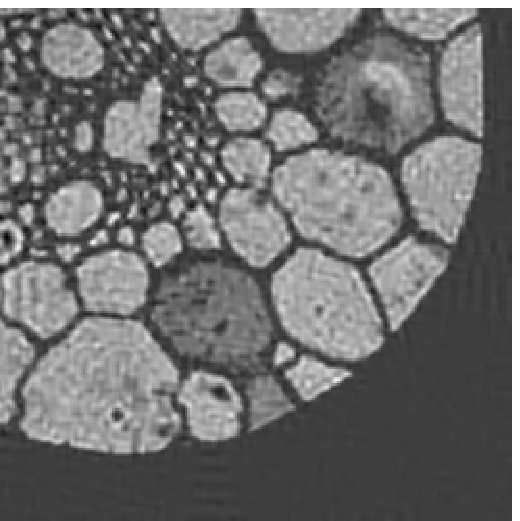}\\(c)
	\end{minipage}
	\caption{Grain test problem reconstructions: (a) ISTA, (b) $\mbox{PISTA}_\Lambda$, (c)  $\mbox{Struct-PISTA}_\Lambda$.}
	\label{fig:grain_rec}
\end{figure}
\paragraph{Satellite}
Our final example is the \texttt{atmosphericBlur30} from the MATLAB toolbox RestoreTool \cite{RestoreTools}. The true image, PSF, and blurred and noisy image are reported in Figures~\ref{fig:satellite}(a), (b), and (c), respectively. Since we know the true image we can estimate the noise level in the image, which is approximately $1\%$. Since this is an astronomical image we impose zero bc's.

From the comparison of the computed results in Table~\ref{tbl:compare} we can see that the $\mbox{Struct-PISTA}_h$  method provides the best reconstruction among all considered methods. We can observe that, in this particular example, ISTA provides a very low quality reconstruction both in term of RRE and SSIM. We report in Figure~\ref{fig:satellite_rec} some reconstructions. From the visual inspection of the computed solutions we can observe that both the approximations obtained with $\mbox{PISTA}_h$ and $\mbox{Struct-PISTA}_h$ do not present heavy ringing effects, while the reconstruction obtained by AIT-GP presents very heavy ringing around the ``arms'' of the satellite. This allows us to show the benefits of introducing the soft-thresholding into the AIT-GP method.
\begin{figure}
	\centering
	\begin{minipage}[c]{0.3\textwidth}
		\centering
		\includegraphics[width=\textwidth]{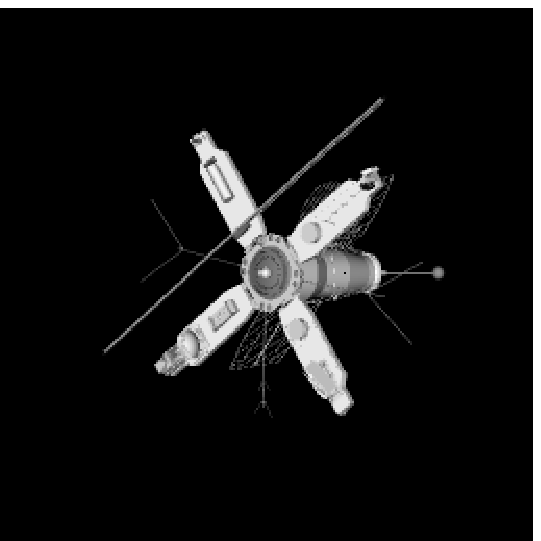}\\(a)
	\end{minipage}
	\begin{minipage}[c]{0.3\textwidth}
		\centering
		\includegraphics[width=\textwidth]{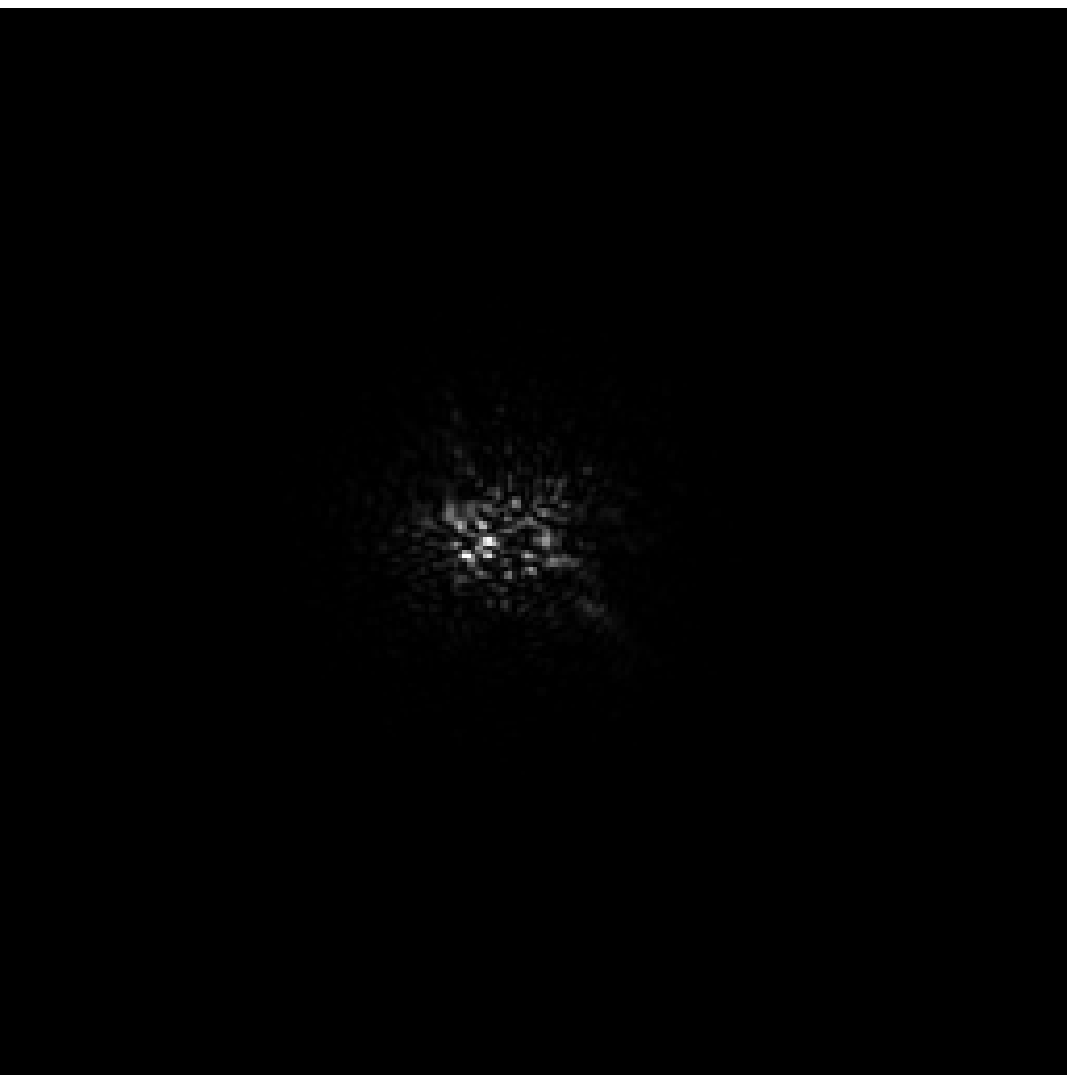}\\(b)
	\end{minipage}
	\begin{minipage}[c]{0.3\textwidth}
		\centering
		\includegraphics[width=\textwidth]{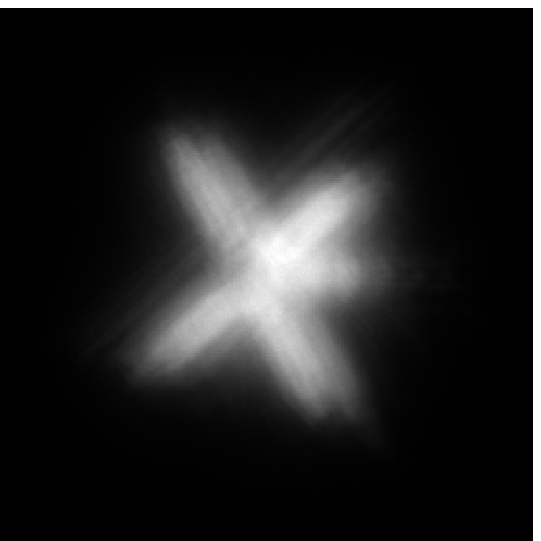}\\(c)
	\end{minipage}
	\caption{Satellite test problem: (a) True image ($256\times 256$ pixels), (b) PSF ($256\times 256$ pixels), (c) Blurred and noisy image with $\approx1\%$ of white Gaussian noise ($256\times 256$ pixels).}
	\label{fig:satellite}
\end{figure}

\begin{figure}
	\centering
	\begin{minipage}[c]{0.3\textwidth}
		\centering
		\includegraphics[width=\textwidth]{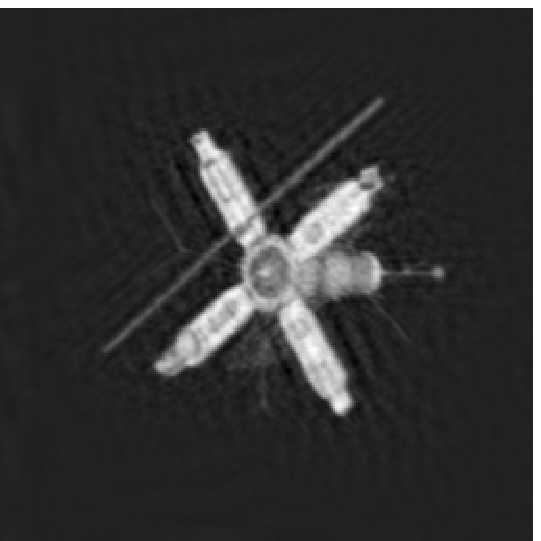}\\(a)
	\end{minipage}
	\begin{minipage}[c]{0.3\textwidth}
		\centering
		\includegraphics[width=\textwidth]{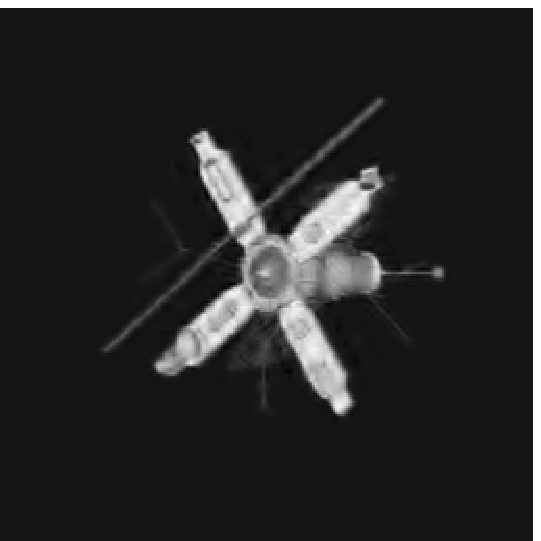}\\(b)
	\end{minipage}
	\begin{minipage}[c]{0.3\textwidth}
		\centering
		\includegraphics[width=\textwidth]{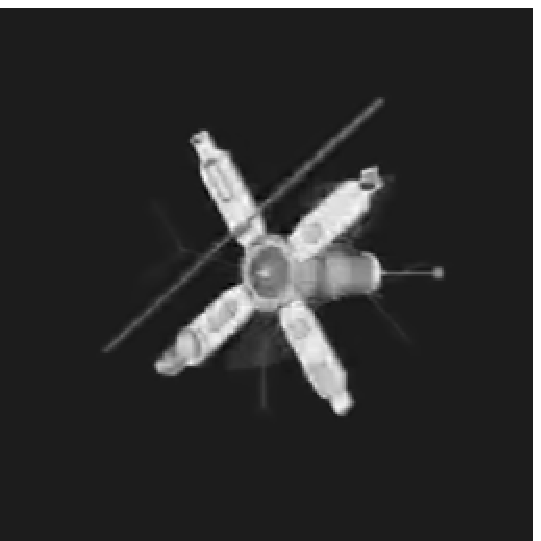}\\(c)
	\end{minipage}
	\caption{Satellite test problem reconstructions: (a) AIT-GP, (b) $\mbox{PISTA}_h$, (c)  $\mbox{Struct-PISTA}_h$.}
	\label{fig:satellite_rec}
\end{figure}
\begin{table}\label{tbl:compare}
	\begin{center}
		\caption{Comparison of the quality of the reconstructions for all considered examples. We highlight in boldface the best result.}
		\begin{tabular}{l|l|lll}
			Example & Method & RRE & PSNR & SSIM\\\hline
			\multirow{6}{*}{Cameraman} & AIT-GP& $0.111024$ & $24.7798$ & $0.729095$\\
			& ISTA& $0.090921$ & $26.5149$ & $0.763217$\\
			& $\mbox{PISTA}_h$& $0.096558$ & $25.9924$ & $0.790363$\\
			& $\mbox{PISTA}_\Lambda$& $0.094853$ & $26.1471$ & $0.795061$\\
			& $\mbox{Struct-PISTA}_h$& ${\bf 0.088796}$ & ${\bf 26.7203}$ & ${\bf 0.840145}$\\
			& $\mbox{Struct-PISTA}_\Lambda$& $0.090182$ & $26.5857$ & $0.834532$\\\hline
			\multirow{6}{*}{Grain} & AIT-GP& $0.183796$ & $25.9571$ & $0.731407$\\
			& ISTA& ${\bf 0.160655}$ & ${\bf 27.1259}$ & $0.845816$\\
			& $\mbox{PISTA}_h$& $0.195516$ & $25.4202$ & $0.737254$\\
			& $\mbox{PISTA}_\Lambda$& $0.181727$ & $26.0554$ & $0.748582$\\
			& $\mbox{Struct-PISTA}_h$& $0.161715$ & $27.0688$ & ${\bf 0.859284}$\\
			& $\mbox{Struct-PISTA}_\Lambda$& $0.168472$ & $26.7133$ & $0.830990$\\\hline
			\multirow{6}{*}{Satellite} & AIT-GP& $0.222783$ & $26.6708$ & $0.742416$\\
			& ISTA& $0.286179$ & $24.4956$ & $0.657111$\\
			& $\mbox{PISTA}_h$& $0.192146$ & $27.9558$ & $0.928584$\\
			& $\mbox{PISTA}_\Lambda$& $0.193730$ & $27.8844$ & $0.916993$\\
			& $\mbox{Struct-PISTA}_h$& ${\bf 0.187970}$ & ${\bf 28.1466}$ & ${\bf 0.934876}$\\
			& $\mbox{Struct-PISTA}_\Lambda$& $0.189147$ & $28.0924$ & $0.924931$\\\hline
		\end{tabular}
	\end{center}
\end{table}

\section{Conclusions}\label{sec:conclusions}
This work develops further and brings together all the techniques studied in \cite{HDC14,B,CDBH,DDEM,BBD}. The idea is to combine thresholding iterative methods, an approximate Tikhonov regularization scheme depending on a general (differential) operator and a structure preserving approach, with the main goal in mind to reduce the boundary artifacts which appear in the resulting de-blurred image when imposing artificial boundary conditions. The numerical results are promising and show improvements with respect to known state-of-the-art deblurring algorithms. There are still open problems, mainly concerning the theoretical assumptions and convergence proofs which will be furtherly investigated in future works.

\appendix

\section{Proofs}\label{appendix}
Hereafter we analyze Algorithm \ref{P=f(CC^*) circulant}, aiming to prove its convergence. Most of the following results are a collection of revised results that appeared in \cite{DH13,CDBH,B}. Since the proofs are very technical, we will present a full treatment leaving no details, in order to make this paper self-contained and easily readable.

Following up Section \ref{sec:def}, we need to set some more notations. Let us consider the singular value decomposition (SVD) of $C$ as the triple $\left(U,V,\Sigma\right)$ such that 
\begin{eqnarray*}
	&C\fb = U\Sigma V^* \fb, \\ 
	&U, V\in \mathcal{O}(m^2,\R), \qquad \Sigma= \textnormal{diag}_{j=1,\cdots,m^2}(\sigma_j) \mbox{ with } 0\leq \sigma_{m^2}\leq \cdots \leq \sigma_1,
\end{eqnarray*}
where $\mathcal{O}(m^2,\R)$ is the orthonormal group and $V^*$ is the adjoint of the operator $V$, i.e., $\langle V\fb^1, \fb^2\rangle = \langle \fb^1, V^*\fb^2\rangle$ for every pair $\fb^1, \fb^2 \in \R^{m^2}$. We will indicate the spectrum of $CC^*$ with
$$
\sigma(CC^*) = \{0\} \cup \bigcup_{j=1}^{m^2} \{\sigma_j^2\}.
$$

Hereafter, without loss of generality we will assume that 
$$
\|C\|=1\qquad \mbox{and} \qquad  \|h(CC^*)\|=\max_{\sigma^2 \in [0,1]} h(\sigma^2)=1.
$$
The first issue we have to consider is the existence of the sequence $\{\alpha_n\}$.

\begin{lemma}
	Let $\|\rb^n\|>\tau\delta$. Then for every fixed $n$ there exists $\alpha_n$ that satisfies \eqref{eq:Alg4_rev_7}. It can be computed by the following iteration 
	\begin{equation}\label{eq:Newton}
	\alpha_n^{k+1} := \frac{\left(\alpha_n^k\right)^2 \Phi'(\alpha_n^k)}{\alpha_n^k \Phi'(\alpha_n^k) + \Phi(\alpha_n^k) - q_n^2\|\rb^n\|},
	\end{equation}
	where
	\begin{align*}
	&\Phi(\alpha) := \| \alpha  (CC^* + \alpha h\left(CC^*\right))^{-1} \rb^n\|^2,\\
	&\Phi'(\alpha) := \| \sqrt{2\alpha}CC^* (CC^* + \alpha h\left(CC^*\right))^{-3/2} \rb^n\|^2.
	\end{align*}
	The convergence is locally quadratic. The existence of the regularization parameter $\alpha_n$ and the locally quadratic convergence of the algorithm above are independent and uniform with respect to the dimension $m^2$.   
\end{lemma}
\begin{proof}
	The existence of $\alpha$ is an easy consequence of the monotonicity of 
	$$
	\phi_\alpha(\sigma^2) = \alpha\left(\sigma^2 + \alpha h(\sigma^2)\right)^{-1}
	$$
	with respect to $\alpha$. Indeed, let us rewrite \eqref{eq:Alg4_rev_7} as follows
	\begin{align}\label{eq:Alg4_rev_7*}
	\alpha^2 \|(CC^* + \alpha h\left(CC^*\right))^{-1} \rb^n\|^2 &= \|\phi_\alpha (CC^*)\rb^n\|^2 \nonumber\\
	&= \int_{[0,1]}\phi^2_\alpha(\sigma^2) d\|\rb^n\|(\sigma) \nonumber\\
	&= \int_{[0,1]} \frac{\alpha^2 }{\left(\sigma^2 +\alpha h(\sigma^2)\right)^2} d\|\rb^n\|(\sigma) \nonumber\\
	& = \frac{q_n^2}{c_1^2} \int_{[0,1]}d\|\rb^n\|(\sigma),
	\end{align}
	where $d\| \rb^n\|(\cdot)$ is the discrete spectral measure associated to $\rb^n$ with respect to the SVD of $C$ and $\sigma \in \sigma(C)$ are the singular values of the spectrum of $C$. Since $\frac{d \phi_\alpha}{d\alpha} >0$ for every $\alpha\geq 0$, then by monotone convergence it holds that
	\begin{align*}
	\lim_{\alpha\to \infty} \int_{[0,1]} \frac{\alpha^2}{\left(\sigma^2 +\alpha h(\sigma^2)\right)^2} d\|\rb^n\|(\sigma) &= \int_{[0,1]} \lim_{\alpha\to \infty} \frac{\alpha^2 }{\left(\sigma^2 +\alpha h(\sigma^2)\right)^2} d\|\rb^n\|(\sigma)\\
	&\geq \|\rb^n\|^2 > \frac{q_n^2}{c_1^2}\|\rb^n\|^2.
	\end{align*}
	Indeed, it is not difficult to prove that $q_n/c_1 <1$ whenever $\rho \in (0,c_1/2)$ and $\|\rb^n\|> \tau\delta$, as assumed in the hypothesis. Since for $\alpha=0$ the left hand-side of \eqref{eq:Alg4_rev_7*} is zero, then we conclude that there exists an unique $\alpha_n >0$ such that equality holds in \eqref{eq:Alg4_rev_7}. Due to the generality of our proof and the fact that we could pass the limit under the sign of integral, the existence of such an $\alpha_n$ is granted uniformly with respect to the dimension $m^2$. 
	
	Since 
	$$
	\phi_\alpha(\sigma^2) = \alpha\left(\sigma^2 + \alpha h(\sigma^2)\right)^{-1} = \left(\alpha^{-1}\sigma^2 + h(\sigma^2)\right)^{-1},
	$$
	fixing $\gamma = \alpha^{-1}$, let us now define the following function
	$$
	\psi_\gamma(\sigma^2) = \left(\gamma\sigma^2 + h(\sigma^2)\right)^{-1}.
	$$
	Since
	\begin{equation}\label{eq:lem1}
	\frac{\partial\psi_\gamma^2 (\sigma^2)}{\partial_\gamma} = -2 \sigma^2\left(\gamma\sigma^2 + h(\sigma^2)\right)^{-3}, \qquad
	\frac{\partial^2\psi_\gamma^2 (\sigma^2)}{\partial_{\gamma^2}} = 6 \sigma^4\left(\gamma\sigma^2 + h(\sigma^2)\right)^{-4},
	\end{equation}
	then there exists two constants $d_1, d_2$ independents of $\gamma$ such that
	$$
	\left|\frac{\partial\psi_\gamma^2 (\sigma^2)}{\partial_\gamma}\right|\leq d_1, \qquad \left|\frac{\partial^2\psi_\gamma^2 (\sigma^2)}{\partial_{\gamma^2}}\right| \leq d_2,
	$$
	and in particular $d_1, d_2 \in \textnormal{L}^1([0,1], d\|\rb^n\|)$ for every $n$ and $m$. Therefore, if we define 
	$$
	\Psi(\gamma):= \|\psi_\gamma (CC^*)\rb^n\|^2,
	$$
	it holds that
	\begin{eqnarray}
	&\Psi'(\gamma) = \frac{\partial}{\partial_\gamma} \int_{[0,1]} \psi_\gamma^2(\sigma^2) d\|\rb^n\|(\sigma) = \int_{[0,1]}  \frac{\partial \psi_\gamma^2(\sigma^2)}{\partial_\gamma} d\|\rb^n\|(\sigma),\label{eq:lem2.1}\\
	&\Psi''(\gamma)=\frac{\partial}{\partial_\gamma} \int_{[0,1]} \frac{\partial \psi_\gamma^2(\sigma^2)}{\partial_\gamma} d\|\rb^n\|(\sigma) = \int_{[0,1]}  \frac{\partial^2 \psi_\gamma^2(\sigma^2)}{\partial_{\gamma^2}} d\|\rb^n\|(\sigma).\label{eq:lem2.2}
	\end{eqnarray}
	Then the Newton iteration applied to $\Psi(\gamma) = q_n^2\|\rb^n\|^2$ yields the iteration
	\begin{align*}
	\gamma_n^{k+1} &= \gamma_n^{k} + \frac{q_n^2\|\rb^n\|^2 - \Psi(\gamma_n^{k})}{\Psi'(\gamma_n^{k})}, \qquad k\geq 0.
	\end{align*}
	By \eqref{eq:lem1}, $\Psi(\gamma)$ is a decreasing convex function in $\gamma$. Since $\gamma_n= \lim_{k\to \infty}\gamma_n^{k+1} = 1/\alpha_n$, obviously we have that
	\begin{equation}\label{eq:lem3}
	\Psi(\gamma_n) = \frac{q_n^2}{c_1^2} \|\rb^n\|^2. 
	\end{equation}
	If 
	$$
	\Psi'(\gamma_n)  = -\|\sqrt{2}(CC^*)(\gamma_n CC^*+h(CC^*))^{-3/2} \rb^n\|^2=0,
	$$
	then necessarily we would have that $CC^*\rb^n =0$. Hence, $\left(\gamma_n CC^* +h(CC^*)\right)^{-1}\rb^n = h(CC^*)\rb^n$, and consequently 
	$$
	\Psi(\gamma_n)  = \| h(CC^*)\rb^n\|^2.
	$$
	From \eqref{eq:lem3} we would deduce that $q_n\geq c_1$, but this is absurd since as already observed above, $q_n<c_1$ if $\|\rb^n\|>\tau\delta$. Therefore $\Psi'(\gamma_n) \neq 0$ and by standard properties of the Newton iteration, $\gamma_n^{k}$ converges to the minimizer $\gamma_n$ from below and the convergence is locally quadratic. Finally, defining
	$$
	\Phi(\alpha) = \Psi(1/\alpha),
	$$
	then we get \eqref{eq:Newton}, $\alpha_n^{k}$ converges monotonically from above to $\alpha_n$ and the convergence is locally quadratic. Again, thanks to \eqref{eq:lem2.1} and \eqref{eq:lem2.2}, the rate of convergence is uniform with respect to the dimension of $\Y$.
\end{proof}

From now on, instead of working with the error $\eb^n _\delta = \xb - \xb^n _\delta$, in order to simplify the following proofs and notations, it is useful to consider the partial error with respect to $\zb^n _\delta$, namely
\begin{equation}\label{eq:tildee}
\tilde{\eb}^n _\delta = \xb - \zb^n _\delta.
\end{equation}
This will not affect the generality of our proofs, thanks to the continuity of $\S_\mu(\cdot)$ with respect to the noise level $\delta$.

\begin{proposition}\label{prop:Alg4_rev_1}
	Under the assumptions \eqref{eq:Alg4_rev_00}, if $\| \rb^n _\delta\| > \tau \delta$ and we define $\tau_n = \| \rb^n_\delta \|/ \delta$, then it follows that
	\begin{equation}\label{eq:Alg4_rev_10}
	\| \rb^n _\delta - B\tilde{\eb}^n _\delta \| \leq \left(  \rho + \frac{1+ 2\rho}{\tau_n}  \right) \| \rb^n _\delta \| 
	< (1-\rho) \| \rb^n _\delta \|,
	\end{equation}
	where $\tilde{\eb}^n$ is defined in \eqref{eq:tildee}.
\end{proposition}
\begin{proof}
	In the free noise case we have $\gb = K\xb$. As a consequence
	\begin{align*}
	\rb^n _\delta - B\tilde{\eb}^n _\delta & = \gb^\delta - K \xb^n _\delta -  B(\xb - \zb^n _\delta) + B \xb^n _\delta - B S_\mu(\zb^n _\delta) \\
	& = \gb^\delta - \gb + (K - B)\eb^n _\delta + B(\zb^n _\delta-S_\mu(\zb^n _\delta)).
	\end{align*}
	Using now assumptions \eqref{eq:Alg4_rev_00}, in particular 
	\eqref{eq:Alg4_rev_1}, and $\| \gb^\delta - \gb \| \leq \delta$, we derive the following estimate
	\begin{align*}
	\| \rb^n _\delta - B\tilde{\eb}^n _\delta \| &\leq  \| \gb^\delta - \gb \| + \| (K - B)\eb^n _\delta  \| + \|B(\zb^n _\delta-S_\mu(\zb^n _\delta))\|\\
	&\leq  \| \gb^\delta - \gb \| + \rho \| K\eb^n _\delta \| +\rho\delta\\
	& \leq  \| \gb^\delta - \gb \| + \rho (\|\rb^n _\delta \| + \| \gb^\delta - \gb \|  + \delta)\\
	&\leq (1 + 2\rho) \delta + \rho \|\rb^n _\delta \|.
	\end{align*}
	The first inequality in \eqref{eq:Alg4_rev_10} now follows from the hypothesis $\delta = \|\rb^n _\delta \|/ \tau_n$.
	The second inequality  follows from $ \rho + \frac{1+ 2\rho}{\tau_n} < \rho + \frac{1+ 2\rho}{\tau}$.
\end{proof}

Combining the preceding proposition with \eqref{eq:Alg4_rev_7}, we are going to show that the sequence $\|\tilde{\eb}^n _\delta\|$ is monotonically decreasing. We have the following result.

\begin{proposition}\label{prop:Alg4_rev_2}
	Let $\tilde{\eb}^n _\delta$ be defined in \eqref{eq:tildee}. 
	If the assumptions \eqref{eq:Alg4_rev_00} are satisfied, then $\|\tilde{\eb}^n _\delta\|$ of Algorithm \ref{P=f(CC^*) circulant} decreases monotonically for $n=0,1,\dots,n_\delta-1$. In particular, 
	we deduce
	
	\begin{equation}\label{eq:Alg4_rev_6}
	\|\tilde{\eb}^n _\delta \|^2 - \| \tilde{\eb}^{n+1} _\delta \|^2 \geq \frac{8\rho^2}{1 + 2\rho} \| (CC^* + \alpha_n h\left(CC^*\right) )^{-1} \rb^n _\delta \|\|\rb^n  _\delta\|>0.
	\end{equation}
\end{proposition}
\begin{proof}
	Recalling that $WC^*=B^*$ and that $BB^*=CC^*$, we have
	\begin{align*}
	\| \tilde{\eb}^n _\delta \|^2 - \| \tilde{\eb}^{n+1} _\delta \|^2 &= 2\langle \tilde{\eb}^n _\delta, \hb^n \rangle - \| \hb^n \|^2\\
	&= 2\langle B\tilde{\eb}^n _\delta,(CC^* + \alpha_n h\left(CC^*\right))^{-1} \rb^n _\delta \rangle - \langle  \rb^n _\delta, CC^* (CC^* + \alpha_n h\left(CC^*\right))^{-2}  \rb^n _\delta \rangle \\
	&= 2 \langle  \rb^n _\delta,(CC^* + \alpha_n h\left(CC^*\right))^{-1}  \rb^n _\delta \rangle -   \langle  \rb^n _\delta, CC^*(CC^* + \alpha_n h\left(CC^*\right))^{-2}  \rb^n _\delta \rangle \\
	&- 2 \langle  \rb^n _\delta - B\tilde{\eb}^n _\delta,(CC^* + \alpha_n h\left(CC^*\right))^{-1}  \rb^n _\delta \rangle\\
	&\geq 2 \langle  \rb^n _\delta, (CC^* + \alpha_n h\left(CC^*\right))^{-1}  \rb^n _\delta \rangle - 2  \langle  \rb^n _\delta, CC^*(CC^* + \alpha_n h\left(CC^*\right))^{-2}  \rb^n _\delta \rangle \\
	&- 2 \langle  \rb^n _\delta - B\tilde{\eb}^n _\delta, (CC^* + \alpha_n h\left(CC^*\right))^{-1}  \rb^n _\delta \rangle\\
	&= 2 \alpha_n \langle  \rb^n _\delta, h\left(CC^*\right)(CC^* + \alpha_n h\left(CC^*\right))^{-2}  \rb^n _\delta \rangle \\
	&- 2 \langle \rb^n _\delta - B\tilde{\eb}^n _\delta, (CC^* + \alpha_n h\left(CC^*\right))^{-1}  \rb^n _\delta \rangle\\
	&\geq  2 \alpha_n \langle  \rb^n _\delta, h\left(CC^*\right) (CC^* + \alpha_n h\left(CC^*\right))^{-2}  \rb^n _\delta \rangle \\
	&- 2\|  \rb^n _\delta - B\tilde{\eb}^n _\delta \| \| (CC^* + \alpha_n h\left(CC^*\right))^{-1}  \rb^n _\delta \|\\
	&\geq 2 \| (CC^* + \alpha_n h\left(CC^*\right))^{-1}  \rb^n _\delta \| \left( \| c_1\alpha_n  (CC^* + \alpha_n h\left(CC^*\right))^{-1}  \rb^n _\delta\|\right.\\
	& - \left. \| \rb^n _\delta - B\tilde{\eb}^n _\delta \|  \right)\\
	&\geq  2 \|(CC^* + \alpha_n h\left(CC^*\right))^{-1}  \rb^n _\delta \|\cdot \left( q_n\|\rb^n _\delta\|  - \left( \rho + \frac{1+ 2\rho}{\tau_n}  \right)\|\rb^n _\delta\|  \right)\\
	&\geq \frac{8\rho^2}{1 + 2\rho}  \| (CC^* + \alpha_n h\left(CC^*\right))^{-1}  \rb^n _\delta \|\|\rb^n _\delta\|>0,
	\end{align*}
	where the relevant inequalities are a consequence of equation \eqref{eq:Alg4_rev_7} and Proposition~\ref{prop:Alg4_rev_1}. 
	The last inequality follows from \eqref{eq:Alg4_rev_7} and $\tau_n > \tau=(1+2\rho)/(1-2\rho)$ for $\| \rb^n _\delta \| > \tau \delta$. 
\end{proof}

\begin{corollary}\label{cor:Alg4_rev_1}
	Under the assumptions \eqref{eq:Alg4_rev_00}, there holds
	\begin{align}\label{eq:Alg4_rev_11}
	\| \tilde{\eb}^0 _\delta \| &\geq  \frac{8\rho^2}{1 + 2\rho}\sum_{n=0}^{n_\delta -1}  \| (CC^* + \alpha_n h\left(CC^*\right))^{-1}  \rb^n _\delta \|\| \rb^n _\delta\| \\
	&\geq c  \sum_{n=0}^{n_\delta -1}\| \rb^n _\delta\|^2 \label{eq:Alg4_rev_12}
	\end{align}
	for some constant $c >0$, depending only on $\rho$ and $q$ in \eqref{eq:Alg4_rev_7}.
\end{corollary}
\begin{proof}
The first inequality follows by taking the sum of the quantities in \eqref{eq:Alg4_rev_6} from $n=0$ up to $n= n_\delta -1$. 
	
	For the second inequality, note that for every 
	$$
	\alpha > \frac{q_n  }{c_1 - q_n}
	$$
	and every $\sigma \in \sigma(C) \subset [0,1]$, we have
	$$
	\frac{\alpha}{\sigma^2 + \alpha h(\sigma^2)} \geq \frac{\alpha}{1 + \alpha} = (1 + 1/\alpha)^{-1} > \frac{q_n}{c_1},
	$$
	and hence,
	$$
	\alpha \|(CC^* + \alpha h\left(CC^*\right))^{-1} \rb^n _\delta \| > \frac{q_n}{c_1} \|  \rb^n _\delta \|,
	$$
	as $\|  \rb^n _\delta \|>0$ for $n < n_\delta$. This implies that  $\alpha_n$ in \eqref{eq:Alg4_rev_7} satisfies $0~< ~\alpha_n ~\leq ~\frac{q_n }{c_1 - q_n}$, thus
	$$
	\|(CC^* + \alpha_n h\left(CC^*\right))^{-1}  \rb^n _\delta \| = \frac{q_n}{c_1\alpha_n} \|\rb^n _\delta \| \geq (c_1 - q_n) \| \rb^n _\delta \|.
	$$
	According to the choice of parameters in Algorithm \ref{P=f(CC^*) circulant}, we deduce 
	$$
	c_1 - q_n = \min \{c_1 - q, c_1 - 2\rho - (1 + \rho)/\tau_n   \},
	$$
	and 
	$$
	c_1 - 2\rho - (1 + \rho)/\tau_n = \frac{1 + 2\rho}{\tau} - \frac{1 + \rho}{\tau_n} > \frac{1 + 2\rho}{\tau} - \frac{1 + \rho}{\tau} = \frac{\rho}{\tau}.
	$$
	Therefore, there exists $c >0$, depending only on $\rho$ and $q$ such that 
	$$
	c_1 - q_n \geq c  \left(\frac{8\rho^2}{1 + 2\rho}\right)^{-1},
	$$
	and
	\begin{align*}
	\| (CC^* + \alpha_n h\left(CC^*\right))^{-1}  \rb^n _\delta \| \geq 
	 c  \left(\frac{8\rho^2}{1 + 2\rho}\right)^{-1} \|  \rb^n _\delta \| \qquad \mbox{for } n= 0, 1, \cdots, n_\delta -1.
	\end{align*}
	Now the second inequality follows immediately.
\end{proof}

From \eqref{eq:Alg4_rev_12} it can be seen that the sum of the squares of the residual norms is bounded, and hence, if $\delta >0$, there must be a first integer $n_\delta < \infty$ such that \eqref{eq:Alg4_rev_12} is fulfilled, i.e., Algorithm \ref{P=f(CC^*) circulant} terminates after finitely many iterations. 


Finally, we are ready to prove a convergence and regularity result.
\begin{theorem}
	Assume that $\zb^0$ is not a solution of the linear system
	\begin{equation}\label{eq:O}
	\gb = A W^* \xb,
	\end{equation}
	and that $\delta_m$ is a sequence of positive real numbers such that $\delta_k \to 0$ as $k \to \infty$. Then, if Assumption \ref{hp:1} is valid, the sequence $\{ \xb^{n(\delta_k)} _{\delta_k}  \}_{k \in \mathbb{N}}$, generated by the discrepancy principle rule \eqref{eq:Alg4_rev_12}, converges as $k \to \infty$ to the solution of \eqref{eq:O} which is closest to $\zb ^0$ in Euclidean norm.
\end{theorem}
\begin{proof}
	We are going to show convergence for the sequence $\{ \zb^{n(\delta_k)} _{\delta_k}  \}_{k \in \mathbb{N}}$ and then the thesis will follow easily from the continuity of $S_{\mu(\delta)}$, i.e.,
	$$
	\lim_{k\to \infty} \xb^{n(\delta_k)} _{\delta_k} = \lim_{k\to \infty}S_{\mu(\delta_k)} (\zb^{n(\delta_k)} _{\delta_k}) =  S_{\lim_{k\to \infty}\mu(\delta_k)} (\lim_{k\to \infty}\zb^{n(\delta_k)} _{\delta_k}) = \lim_{k\to \infty}\zb^{n(\delta_k)} _{\delta_k}.
	$$
	The proof of the convergence for the sequence $\{ \zb^{n(\delta_k)} _{\delta_k}  \}$ can be divided into two steps: at step one, we show the convergence in the free noise case $\delta =0$. In particular, the sequence $\{ \zb^n \}$ converges to a solution of \eqref{eq:O} that is the closest to $\zb ^0$. At the second step, we show that given a sequence of positive real numbers $\delta_k \to 0$ as $k \to \infty$, then we get a corresponding sequence $\{\zb ^{n(\delta_k)} _{\delta_k} \}$ converging as $k \to \infty$.
	
	\textbf{Step 1:} Fix $\delta =0$. It follows that $\rb^n _\delta = \rb^n$, and the sequence $\{ \zb ^n \}$ will not stop, i.e., $n \to \infty$, since the discrepancy principle will not be satisfied by any $n$, in particular $n_\delta \to \infty$ for $\delta \to 0$. Set $n>l>j$, with $n,l,j \in \mathbb{N}$. It holds that
	\begin{align}
	\|\zb^n - \zb^l\|^2 &= 	\|\tilde{\eb}^n - \tilde{\eb}^l\|^2 \nonumber\\
	&=\| \tilde{\eb}^n\|^2 -  \| \tilde{\eb}^l\|^2 - 2 \langle \tilde{\eb}^l, \tilde{\eb}^n - \tilde{\eb}^l \rangle \nonumber\\
	&= \| \tilde{\eb}^n\|^2 -  \| \tilde{\eb}^l\|^2 + 2 \langle \tilde{\eb}^l, \zb^n - \zb^l \rangle \nonumber\\
	&= \| \tilde{\eb}^n\|^2 -  \|\tilde{\eb}^l\|^2 + 2 \sum_{i=l}^{n-1}\langle \tilde{\eb}^l, \hb^i \rangle \nonumber\\
	&= \| \tilde{\eb}^n\|^2 -  \| \tilde{\eb}^l\|^2 + 2 \sum_{i=l}^{n-1}\langle B\tilde{\eb}^l,(CC^* + \alpha_i h\left(CC^*\right))^{-1} \rb^i \rangle \nonumber\\
	&\leq \| \tilde{\eb}^n\|^2 -  \| \tilde{\eb}^l\|^2 + 2 \sum_{i=l}^{n-1} \| B\tilde{\eb}^l\| \|(CC^* + \alpha_i h\left(CC^*\right))^{-1} \rb^i\| \nonumber\\
	&\leq \| \tilde{\eb}^n\|^2 -  \| \tilde{\eb}^l\|^2 + 2(1+\rho) \sum_{i=l}^{n-1} \| \rb^l\| \|(CC^* + \alpha_i h\left(CC^*\right))^{-1} \rb^i\| \label{eq:3.1},
	\end{align}
	where the last inequality comes from \eqref{eq:2c'}. At the same time, we have that
	\begin{align}
	\|(\zb^l - \zb^k)\|^2 &= 	\|(\tilde{\eb}^l - \tilde{\eb}^k)\|^2 \nonumber\\
	&=\| \tilde{\eb}^k\|^2 -  \| \tilde{\eb}^l\|^2 + 2 \langle \tilde{\eb}^l, \tilde{\eb}^k - \tilde{\eb}^l \rangle \nonumber\\
	&= \| \tilde{\eb}^k\|^2 -  \| \tilde{\eb}^l\|^2 - 2 \langle \tilde{\eb}^l, \zb^l - \zb^k \rangle \nonumber\\
	&= \|\tilde{\eb}^k\|^2 -  \| \tilde{\eb}^l\|^2 - 2 \sum_{i=k}^{l-1}\langle \tilde{\eb}^l, \hb^i \rangle \nonumber\\
	&\leq \| \tilde{\eb}^k\|^2 -  \| \tilde{\eb}^l\|^2 + 2 \sum_{i=k}^{l-1} \| B\tilde{\eb}^l\| \|(CC^* + \alpha_i h\left(CC^*\right))^{-1} \rb^i\| \nonumber\\
	&\leq \| \tilde{\eb}^k\|^2 -  \| \tilde{\eb}^l\|^2 + 2(1+\rho) \sum_{i=k}^{l-1} \| \rb^l\| \|(CC^* + \alpha_i h\left(CC^*\right))^{-1} \rb^i\| \label{eq:3.2}.
	\end{align}
	Combining together \eqref{eq:3.1} and \eqref{eq:3.2}, we obtain that
	\begin{align*}
	\|\zb^n - \zb^k\|^2 &\leq 2 	\|\zb^n - \zb^l\|^2 + 	2\|\zb^l - \zb^k\|^2\\
	&\leq 2\| \tilde{\eb}^n\|^2 + 2\| \tilde{\eb}^k\|^2 - 4\| \tilde{\eb}^l\|^2+ 4 (1+\rho) \sum_{i=k}^{n-1} \| \rb^l\| \|(CC^* + \alpha_i h\left(CC^*\right))^{-1} \rb^i\|.
	\end{align*}
	This is valid for every $l \in \{k+1, \cdots, n-1\}$. Choosing $l$ such that $\|\rb^l\|=\min_{i=k+1,\cdots, n-1}\|\rb^i\|$, it follows that
	\begin{align*}
	\|\zb^n - \zb^k\|^2 \leq 2\| \tilde{\eb}^n\|^2 + 2\| \tilde{\eb}^k\|^2 - 4\| \tilde{\eb}^l\|^2+ 4 (1+\rho) \sum_{i=k}^{n-1} \| \rb^i\| \|(CC^* + \alpha_i h\left(CC^*\right))^{-1} \rb^i\|.
	\end{align*}
	
	From Proposition \ref{prop:Alg4_rev_2}, $\{\|\tilde{\eb}^j\|^2\}_{j\in\mathbb{N}}$ is a converging sequence, and from Corollary \ref{cor:Alg4_rev_1} 
	$$
	\sum_{i=k}^{n-1} \| \rb^l\| \|(CC^* + \alpha_i h\left(CC^*\right))^{-1} \rb^i\| \to 0 \qquad \mbox{as }k, n \to \infty
	$$
	since it is the tail of a converging series. Therefore, 
	$$
	\|\zb^n - \zb^k\|^2 \to 0 \qquad \mbox{as }k, n \to \infty
	$$
	and $\{ \zb^n\}_{n\in\mathbb{N}}$ is a Cauchy sequence, and then convergent.
	
	\textbf{Step 2:} Let $\xb$ be the converging point of the sequence $\{ \zb ^n \}_{n\in\mathbb{N}}$ and let $\delta_k >0$ be a sequence of positive real numbers converging to $0$. For every $\delta_k$, let $n = n(\delta_k)$ be the first positive integer such that \eqref{eq:Alg4_rev_12} is satisfied, whose existence is granted by Corollary \ref{cor:Alg4_rev_1}, and let $\{ \zb^{n(\delta_k)} _{\delta_k} \}$ be the corresponding sequence. For every fixed $\epsilon >0$, there exists $\overline{n} = \overline{n}(\epsilon)$ such that 
	\begin{equation}\label{eq:Alg4_rev_13}
	\| \xb - \zb ^n \| \leq \epsilon /2 \qquad \mbox{for every } n > \overline{n}(\epsilon),
	\end{equation} 
	and there exists $\overline{\delta}= \overline{\delta}(\epsilon)$ for which
	\begin{equation}\label{eq:Alg4_rev_14}
	\| \zb ^{\overline{n}} - \zb^{\overline{n}} _\delta \| \leq \epsilon/2 \qquad \mbox{for every } 0 < \delta < \overline{\delta},
	\end{equation}
	due to the continuity of the operator $\gb \mapsto \zb^n$ for every fixed $n$. Therefore, let us choose $\overline{k}= \overline{k}(\epsilon)$ large enough such that $\delta_k < \overline{\delta}$ and such that $n(\delta_k) > \overline{n}$ for every $k > \overline{k}$. Such $\overline{k}$ does exists since $\delta_k \to 0$ and $n_\delta \to \infty$ for $\delta \to 0$. Hence, for every $k > \overline{k}$, we have
	\begin{align*}
	\| \xb - \zb^{n(\delta_k)} _{\delta_k} \| &= \| \tilde{\eb} ^{n(\delta_k)} _{\delta_k} \| \\
	&\leq  \| \tilde{\eb} ^{\overline{n}} _{\delta_k} \| \\
	&= \| \xb - \zb^{\overline{n}} _{\delta_k} \| \\
	&\leq \| \xb - \zb^{\overline{n}} \|  +  \| \zb^{\overline{n}} - \zb^{\overline{n}} _{\delta_k} \| \leq \epsilon,
	\end{align*}
	where the first inequality comes from Proposition \ref{prop:Alg4_rev_2} and the last one from \eqref{eq:Alg4_rev_13} and \eqref{eq:Alg4_rev_14}.
\end{proof}

\end{document}